\documentclass[a4paper, english, 11pt, twoside, reqno]{amsart}

\usepackage{babel}
\usepackage[T1]{fontenc}
\usepackage{amssymb}
\usepackage{amsmath, amsfonts, mathrsfs, amscd}
\usepackage{calligra}
\usepackage{hyperref}
\usepackage{tikz-cd}

\newtheorem*{theorem-intro}{Theorem}
\newtheorem*{question-intro}{Question}
\newtheorem*{corollary-intro}{Corollary}
\newtheorem{theorem}{Theorem}[section]
\newtheorem{lemma}[theorem]{Lemma}

\newtheorem{corollary}[theorem]{Corollary}

\theoremstyle{definition}
\newtheorem{definition}[theorem]{Definition}
\newtheorem{remark}[theorem]{Remark}

\newcommand{\C}{\mathcal{C}}
\newcommand{\m}{\mathcal{M}}
\newcommand{\g}{\mathfrak{g}}

\newcommand{\Nn}{\mathcal{N}}
\newcommand\trans{\circledast}

\newcommand{\N}{{\mathbb N}}
\newcommand{\Z}{\mathbb{Z}}
\newcommand{\D}{\mathcal{D}}
\newcommand\id{\operatorname{id}}

\allowdisplaybreaks

\begin{document}
\title[A category viewpoint for translation functors for finite $W$-algebras]{A monoidal category viewpoint for translation functors for finite $W$-algebras}

\author[E. Masut]{Elisabetta Masut}

\address{
Dipartimento di Matematica Tullio Levi-Civita,
Universit\`a degli Studi di Padova,
via Trieste 63, 35121 Padova, Italia \bigskip}

\email{elisabetta.masut@gmail.com}
\begin{abstract}
We re-interpret Goodwin’s translation functors for a finite $W$-algebra $H_\ell$ as an action of a  monoidal subcategory of $U(\mathfrak{g})$-mod on the category of finitely generated $H_\ell$-modules. This action is obtained by transporting the tensor product of $U(\mathfrak{g})$-modules through Skryabin's equivalence. We apply this interpretation to show that the Skryabin equivalence by stages introduced by Genra and Juillard is an equivalence of $U(\g)$-module categories.
\end{abstract}
\maketitle
\section{Introduction}
A finite $W$-algebra $H_\ell$ is an algebra constructed from a reductive Lie algebra $\g$ and a nilpotent element $e \in \g$.\\ The first definition of such an algebra dates back to 1978, when Kostant constructed in \cite{K} the algebra $H_\ell$ starting from a regular nilpotent element $e_{\text{reg}}$ and showed that $H_\ell \simeq Z(U(\g))$. \\
In \cite{P}, Premet gave a general definition of a finite $W$-algebra, i.e. he generalized Konstant's construction to the case of a general nilpotent element. In particular, when $e=0$, then $H_\ell \simeq U(\g)$, and when $e$ is regular, then $H_\ell \simeq Z(U(\g))$. The work in \cite{P} was motivated by the study of representations of semisimple Lie algebras in positive characteristic.\\
Roughly speaking, a finite $W$-algebra is a subquotient of $U(\g)$ which lies between $Z(U(\g))$ and $U(\g)$, although in general it is not a subalgebra of $U(\g)$. Further contributions to the comprehension of this construction are due to Gan, Ginzburg, Brundan and Goodwin (\cite{BG}, \cite{BGK}, \cite{GG}).\\ \\
\par
Finite $W$-algebras captured the attention of mathematicians, but also of physicists. The latter were interested in their connection with affine $W$-algebras, which are vertex algebras modeling the so called $W$-Symmetry from conformal field theory (see for instance \cite{BE}).\\
Mathematicians are interested in the representation theory of finite $W$-algebras, because important information about $U(\g)$-modules and primitive ideals of $U(\g)$ are encoded in the representation theory of $H_\ell$ (\cite{L}). An important connection was illustrated by Skryabin in the appendix of \cite{P}, where he established an equivalence between finitely-generated $H_\ell$-modules and a specific subcategory of $U(\g)$-modules. By means of this equivalence, Premet in \cite{P2} proved that every finite $W$-algebra possesses finite dimensional irreducible representations.
\par
In representation theory of Lie algebras, a crucial role is played by the Bernstein-Gelfand-Gelfand category $\mathcal{O}$ (\cite[Chapter 1]{H}). This category contains important modules such as  finite-dimensional modules, highest-weight modules, Verma modules.\\ The category $\mathcal{O}$ is the direct sum of the subcategories $\mathcal{O}_\chi$, called blocks, where $\chi$ runs through the set of central characters of $U(\g)$. These blocks are related to each other by means of translation functors; these functors allow us to deduce equivalences between blocks and to understand how representations behave when they are translated from a block to another one in the BGG category $\mathcal{O}$ (\cite[Chapter 7]{H}).  \\\\
In \cite{G}, Goodwin introduced an analogous functor for $H_\ell$-modules. He defined the translation of a finitely generated $H_\ell$-module by using Skryabin's equivalence. Additionally, he studied the relation between the translation functors and the  category $\mathcal{O}(e)$, which is the analogue of the BGG category in $W$-algebras setting (for the definition see \cite[Subsection 4.4]{BGK}). \\
In this paper, after spelling out the transport of structure procedure for general actions of monoidal categories, we interpret this translation functor as an action of a specific subcategory of $U(\g)$ representations on the category of finitely generated $H_\ell$-modules via transport of structure through Skryabin equivalence.\\
Next, we apply this interpretation to the context of reduction by stages, that we now explain.\\
By Poincaré-Birkhoff-Witt theorem, the universal enveloping algebra $U(\g)$ is a quantization of the symmetric algebra $S(\g)$. Also, in \cite{P}, Premet showed that $H_\ell$ is a quantization of the Slodowy slice $\mathscr{S}_e$ associated to a nilpotent element $e \in \g$. For constructing $\mathscr{S}_e$ one needs a nilpotent element $e \in \g$, which, by Jacobson-Morozov Theorem, can be embedded in an $\mathfrak{sl}_2$-triple $\{e, h, f \}$ i.e. $h, f \in \g$ are such that $[h,e]=2e$, $[h,f]=-2f$ and $[e,f]=h$. Then the Slodowy slice  is defined as $\mathscr{S}_e=e+\g^f$, where $\g^f$ stands for the centralizer of $f$ in $\g$. Since $\g^*$ is a Poisson variety, the Slodowy slice inherits its Poisson structure through Hamiltonian reduction. Hence, the finite $W$-algebra can be seen as a quantum Hamiltonian reduction of $U(\g)$.\\ \\
Since $U(\g) \simeq H_\ell$ when $e=0$, one might ask whether a finite $W$-algebra $H_{\ell_{(2)}}$ can be expressed as a quantum Hamiltonian reduction of another finite $W$-algebra $H_{\ell_{(1)}}$ in such a way that the following diagram commutes.\\
\begin{center}
\begin{tikzcd}[column sep=6em,row sep=5em]
U(\g) \arrow[rr,dashed,"Q.H.R."] \arrow[dr,dashed, "Q.H.R."] && H_{\ell_{(2)}} \\
 & H_{\ell_{(1)}} \arrow[ur,dashed," 'Q.H.R' "] & 
\end{tikzcd}
\end{center}
The theory of Hamiltonian reduction by stages is a well-developed branch of symplectic geometry. The problem is to find a quantum version of Hamiltonian reduction by stages.\\
Firstly Morgan in \cite{M}, and then Genra and Julliard in \cite{GJ} worked on this topic. In particular, the latter were interested in this reduction by stages for affine $W$-algebras and for finite $W$-algebras; for the first one they made some conjectures, while for the finite ones they found some conditions for which a quantum version of Hamiltonian reduction by stages can be performed. \\
As a consequence of this construction, they obtained a variant of Skryabin equivalence. \\
In this paper we show that the above equivalence is compatible with the category action by translation functors.  
\subsection{Organization}
We briefly outline how this paper is organized.\\ \\
In the second Section, we recall some notions about category theory. In particular, we  recall the definition of a $\C$-module category and of a $\C$-module functor where $\C$ is a monoidal category. Then, we spell out how we can construct new $\C$-module categories by means of an equivalence and transport of structure. \\
Moreover, we recall the definition of a finite $W$-algebra and of the classical version of Skryabin equivalence. Also, we briefly recall from \cite{GJ} how reduction by stages for finite $W$-algebras can be performed and then we will present the variant of Skryabin equivalence.\\ \\
The goal of Section 3 is to spell out how the translation functors introduced by Goodwin are an instance of transport of structure of a natural categorical action. To this aim, we show that the category of Whittaker modules is a $(\C_e,\C_e)$-bimodule category, where $\C_e$ is a subcategory of the category of $U(\g)$-modules, depending on $e$. Then, by means of Skryabin equivalence, we endow the category of $H_\ell$-modules with a $(\C_e,\C_e)$-bimodule structure.\\\\
In Section 4, we show that the Skryabin equivalence by stages is invariant under the action of $\C_e$.\\ \\
Finally, in the Appendix, we recollect some technical details needed in Section 1. 
\subsection{Acknowledgements}
The author would like to thank Giovanna Carnovale for her numerous comments and suggestions, which greatly contributed to improve the quality of this paper.
\section{Preliminaries}
In this section, we fix notations and we give preliminary definitions and results on $\C$-module categories and on $W$-algebras.
\subsection{$\C$-modules categories} First of all, we will refer to a monoidal category as a quadruple $(\C, \otimes_\C , a, 1)$, where $\C$ is a category, $\otimes_\C$ is the tensor product bifunctor, $a$ is the associativity constraint and $1$ is the unit object. We do not include the left and the right unit constraints in the notation since they won't play a role in this paper.\\ Also, given a monoidal category $(\C, \otimes_\C , a, 1)$, the opposite monoidal category to $\C$ is denoted by  $(\C^{\text{op}},\otimes_\C^{\text{op}}, a^{\text{op}},1)$, where $\C^\text{op}=\C$, $X \otimes_\C^{\text{op}} Y:=Y \otimes_\C X$ and $a_{X,Y,Z}^{\text{op}}:=a_{Z,Y,X}^{-1}$.\\
The definition of a left $\C$-module category is the categorification of the left action of a monoid on a set. 
\begin{definition}\cite[Definition 7.1.1]{E}
Let $(\C, \otimes_\C, a, 1)$ be a monoidal category.
A left module category over $\C$ is a category $\m$ equipped with an action bifunctor $\otimes_\m \colon \C \times \m \to \m$ and a natural isomorphism 
\begin{equation}
\label{ass}
m_{X,Y,M} \colon (X \otimes_\C Y) \otimes_\m M  \xrightarrow[]{\sim} X \otimes_\m (Y \otimes_\m M), \qquad X,Y \in \C, \, M \in \m,
\end{equation}
called module associativity constraint such that the functor $1 \otimes_\m - \colon \m \to \m $, given by $1 \otimes_\m M \mapsto  M $ is an autoequivalence, and the pentagon diagram:
\begin{center}
\small
$$
\label{penta}
\begin{tikzpicture}[thick,scale=0.6, every node/.style={scale=0.9}]
  \matrix (m) [matrix of math nodes,row sep=2em,column sep=2em,minimum width=2em]
  {
   {\,}  &((X \otimes_\C Y)\otimes_\C Z) \otimes_\m M &{\,} \\
   (X \otimes_\C (Y \otimes_\C Z)) \otimes_\m M&{\,}&(X \otimes_\C Y) \otimes_\m (Z \otimes_\m M)\\
X \otimes_\m ((Y \otimes_\C Z)\otimes_\m M) & {\,} &X \otimes _\m(Y \otimes_\m (Z \otimes_\m M)) \\  };
  \path[-stealth]
    (m-1-2) edge node [above] {$\scriptstyle{a_{X,Y,Z}\otimes_\m \text{id}_M \qquad \qquad}$} (m-2-1)
    (m-1-2) edge node [above] {$\, \, \hspace{0.5cm}\scriptstyle{m_{X \otimes_\C Y,Z,M}}$} (m-2-3)
    (m-2-1) edge node [left] {$\scriptstyle{m_{X,Y\otimes_\C Z, M}}$} (m-3-1)
(m-3-1) edge node [above] {$\scriptstyle{\text{id}_X \otimes_\m m_{Y,Z,M}}$} (m-3-3)
(m-2-3) edge node [right] {$\scriptstyle{m_{X,Y,Z \otimes_ \m M}}$} (m-3-3);
\end{tikzpicture}
$$
\end{center}
is commutative for any $X,Y,Z \in \C$ and $M \in \m$.
\end{definition}

In a similar way one can define a right $\C$-module category $(\m, \otimes^\m, m^r)$. Namely, a right $\C$ module category is the same thing as a left $\C^{\text{op}}$-module category. \\ 
We now introduce the notion of a bimodule category over a pair of monoidal categories.
\begin{definition}\cite[Definition 7.1.7]{E}
Let $(\C, \otimes_\C, a, 1)$ and 	$(\mathcal{D}, \otimes_{\mathcal{D}}, \tilde{a}, \tilde{1})$ be monoidal categories. A $(\C,\mathcal{D})$-bimodule category is a category $\m$ which is a left  $\C$-module category $(\m,\otimes_\m, m)$ and a right $\mathcal{D}$-module category $(\m, \otimes^\m, m^r)$ with module associativity constraints for $X, Y \in \C$, $M \in \m$ and $W, Z \in \mathcal{D}$
$$m_{X,Y,M} \colon (X \otimes_\C Y) \otimes_\m M \overset{\sim}{\to} X \otimes_\m (Y \otimes_\m M)$$ and $$m^r_{M,W,Z}\colon M \otimes^\m (W \otimes_\mathcal{D} Z) \overset{\sim}{\to} (M \otimes^\m W) \otimes^\m Z$$ respectively, compatible by a collection of natural isomorphisms for $X \in \C$, $M \in \m$ and $Z \in \mathcal{D}$ $$b_{X,M,Z}\colon (X \otimes_\m M) \otimes^\m Z \overset{\sim}{\to} X \otimes_\m (M \otimes^\m Z)$$ called middle associativity constraints such that the diagrams 
\small
\begin{center}
\begin{equation}
\label{bim1}
\begin{tikzpicture}[thick,scale=0.6, every node/.style={scale=0.9}]
  \matrix (m) [matrix of math nodes,row sep=2em,column sep=1.5em,minimum width=2em]
  {
   {\,}  &((X \otimes_\C Y)\otimes_\m M) \otimes^\m Z &{\,} \\
   (X \otimes_\m (Y \otimes_\m M)) \otimes^\m Z&{\,}&(X \otimes_\C Y) \otimes_\m(M \otimes^\m Z)\\
X \otimes_\m ((Y \otimes_\m M)\otimes^\m Z) & {\,} &X \otimes_\m (Y \otimes_\m (M \otimes^\m Z)) \\  };
  \path[-stealth]
    (m-1-2) edge node [above] {$\scriptstyle{m_{X,Y,M}\otimes^\m \text{id}_Z \qquad \qquad}$} (m-2-1)
    (m-1-2) edge node [above] {$\scriptstyle{\qquad b_{X \otimes_\C Y,M,Z}}$} (m-2-3)
    (m-2-1) edge node [left] {$\scriptstyle{b_{X,Y\otimes_\m M, Z}}$} (m-3-1)
(m-3-1) edge node [above] {$\scriptstyle{\text{id}_X \otimes_\m b_{Y,M,Z}}$} (m-3-3)
(m-2-3) edge node [right] {$\scriptstyle{m_{X,Y,M \otimes^\m Z}}$} (m-3-3);
\end{tikzpicture}
\end{equation}
\end{center}
and
\begin{center}
\begin{equation}
\label{bim2}
\begin{tikzpicture}[thick,scale=0.6, every node/.style={scale=0.78}]
  \matrix (m) [matrix of math nodes,row sep=2em,column sep=1.5em,minimum width=2em]
  {
   {\,}  &X \otimes_\m (M\otimes^\m ( W \otimes_\mathcal{D} Z)) &{\,} \\
   X \otimes_\m((M \otimes^\m W) \otimes^\m Z)&{\,}&(X \otimes_\m M) \otimes^\m (W \otimes_\mathcal{D} Z)\\
(X \otimes_\m (M \otimes^\m W))\otimes^\m Z & {\,} &((X \otimes_\m M) \otimes^\m W) \otimes^\m Z \\  };
  \path[-stealth]
    (m-1-2) edge node [above] {$\scriptstyle{ \text{id}_X \otimes_\m m^r_{M,W,Z} \qquad \qquad}$} (m-2-1)
    (m-2-3) edge node [above] {$\scriptstyle{\qquad b_{X, M, W \otimes_\mathcal{D}Z}}$} (m-1-2)
    (m-3-1) edge node [left] {$\scriptstyle{b_{X,M\otimes^\m W, Z}}$} (m-2-1)
(m-3-3) edge node [above] {$\scriptstyle{ b_{X,M,W} \otimes^\m \text{id}_Z}$} (m-3-1)
(m-2-3) edge node [right] {$\scriptstyle{m^r_{X \otimes_\m M,W,Z}}$} (m-3-3);
\end{tikzpicture}
\end{equation}
\end{center}
commute for all $X,Y \in \C, Z, W \in \mathcal{D}$ and $M \in \m$.
\end{definition}

An important notion we need to categorify is the notion of equivariant morphisms. This leads to the following definition. 
\begin{definition}\cite[Definition 7.2.1]{E}
\label{modulefunctor}
Let $(\m, \otimes_\m, m)$ and $(\mathcal{N}, \otimes_{\mathcal{N}},n)$ be two module categories over a monoidal category $(\C, \otimes_\C, a, 1)$ with associativity constraints $m$ and $n$, respectively. A $\C$-module functor from $\m$ to $\mathcal{N}$ consists of a functor $F \colon \m \to \mathcal{N}$ and a natural isomorphism 
\begin{equation}
\label{natisofunctor}
s_{X,M} \colon F(X \otimes_\m M) \to X \otimes_\mathcal{N} F(M), \qquad X \in \C,\, M \in \m
\end{equation}
such that the following diagrams 
\small
\begin{center}
\begin{equation}
\label{functor1}
\begin{tikzpicture}
 \matrix (m) [matrix of math nodes,row sep=2em,column sep=2em,minimum width=2em]
 {
  {\,}  & F((X\otimes_\C Y) \otimes_\m M) &{\,} \\
F(X \otimes_\m (Y \otimes_\m M)&{\,}&(X \otimes_\C Y) \otimes_\mathcal{N} F(M)\\
X \otimes_\mathcal{N} F(Y \otimes_\m M) & {\,} &X \otimes_\mathcal{N}(Y \otimes_\mathcal{N} F(M)) \\  };
\path[-stealth]
(m-1-2) edge node [above] {$\scriptstyle{F(m_{X,Y,M}}) \hspace{2em},$} (m-2-1)
(m-1-2) edge node [above] {$\hspace{2em}\scriptstyle{s_{X \otimes_\C Y,M}}$} (m-2-3)
(m-2-1) edge node [left] {$\scriptstyle{s_{X,Y\otimes_\m M}}$} (m-3-1)
(m-3-1) edge node [above] {$\scriptstyle{\text{id}_X \otimes_\mathcal{N} s_{Y,M}}$} (m-3-3)
(m-2-3) edge node [right] {$\scriptstyle{n_{X,Y,F(M)}}$} (m-3-3);
\end{tikzpicture}
\end{equation}
\end{center}
and
\begin{center}
\begin{equation}
\label{functor2}
\begin{tikzpicture}
 \matrix (m) [matrix of math nodes,row sep=2em,column sep=5em,minimum width=2em]
 {
  F(1 \otimes_\m M) &{\,}& 1 \otimes_\mathcal{N} F(M)\\
 {\,}&{\,}&{\,}\\
{\,} & F(M) & {\,} \\ };
 \path[-stealth]
 (m-1-1) edge node [above] {$\scriptstyle{s_{1,M}}$} (m-1-3)
(m-1-1) edge node [above] {$\scriptstyle{\qquad \textcolor{white}{F(L_M)}}$} (m-3-2)
(m-1-3) edge node [above] {$\scriptstyle{\textcolor{white}{\tilde{L}_{F(M)}}\qquad}$} (m-3-2);
\end{tikzpicture}
\end{equation}
\end{center}
commute for all $X,Y \in \C$ and $M \in \m$.
\end{definition}
\begin{definition}
Let $(\m, \otimes_\m, m)$ and $(\mathcal{N}, \otimes_{\mathcal{N}},n)$ be two module categories over a monoidal category $(\C, \otimes_\C, a, 1)$ with associativity constraints $m$ and $n$, respectively.  We say that a $\C$-module functor $(F \colon \m \to \Nn, s)$ is an equivalence of $\C$-module categories if $F$ is an equivalence.
\end{definition}
The above definitions have a specular counterpart for right-module categories.
For the definition of a bimodule functor, we also need a compatibility condition. In particular, \cite[Definition 2.10]{Green} together with \cite[Remark 2.14]{Green} gives us the following:
\begin{definition}
Let  $(\C, \otimes_\C, a, 1)$ and $(\mathcal{D}, \otimes_\mathcal{D}, \tilde{a}, \tilde{1})$ be two monoidal categories and let $(\m, \otimes_\m, \otimes^\m, m, m^r, b)$ and $(\mathcal{N}, \otimes_{\mathcal{N}}, \otimes^{\mathcal{N}},n, n^r, p)$ be two $(\C, \mathcal{D})$-bimodule categories. A $(\C, \mathcal{D})$-bimodule functor from $\m$ to $\Nn$ consists of a triple $(F, s, s^r)$, where $F$ is a functor from $\m$ to $\Nn$, $(F,s)$ is a left $\C$-module functor and $(F,s^r)$ is a right $\C$-module functor, such that the following diagram commutes
\begin{center}
\begin{equation}
\label{functorbimodule}
\begin{tikzpicture}
 \matrix (m) [matrix of math nodes,row sep=4em,column sep=2.5em,minimum width=2em]
 {
  F((X \otimes_\m M) \otimes^\m Y)  &\, & F(X \otimes_\m (M \otimes^\m Y)) \\
 F(X\otimes_\m M) \otimes^\Nn Y&{\,}& X \otimes_\Nn F(M \otimes^\m Y)\\
(X \otimes_\Nn F(M)) \otimes^\Nn Y& {\,} & X \otimes_\Nn (F(M) \otimes^\Nn Y) \\  };
\path[-stealth]
(m-1-1) edge node [above] {$\scriptstyle{F(b_{X,M,Y})}$} (m-1-3)
(m-1-1) edge node [left] {$\scriptstyle{s^r_{X \otimes_\m M, Y}}$} (m-2-1)
(m-2-1) edge node [left] {$\scriptstyle{s_{X,M} \otimes^\Nn \text{id}_Y}$} (m-3-1)
(m-3-1) edge node [above] {$\scriptstyle{p_{X, F(M), Y}}$} (m-3-3)
(m-1-3) edge node [right] {$\scriptstyle{s_{X, M \otimes^\m Y}}$} (m-2-3)
(m-2-3) edge node [right] {$\scriptstyle{\text{id}_X \otimes_\Nn s^r_{M,Y}}$} (m-3-3);
\end{tikzpicture}
\end{equation}
\end{center}
for every $M \in \m$, $X \in \C$ and $Y \in \mathcal{D}$.
\end{definition}
We conclude this subsection with a property about $\C$-module functors.
\begin{lemma}
\label{compositionfunctor}
Let $(\m, \otimes_\m, m), (\Nn, \otimes_\Nn, n)$ and $(\mathcal{P}, \otimes_\mathcal{P}, p)$ be module categories over a monoidal category $(\C, \otimes_\C, a, 1)$. Let $(F \colon \m \to \Nn, s)$ and $(G \colon \Nn \to \mathcal{P},t)$ be two $\C$-module functors. Then $(G \circ F, u:= t_{-,F(-)} \circ G(s))$ is a $\C$-module functor.
\end{lemma}
\begin{proof}
Firstly, observe that $u$ has the required source and target and it is a natural isomorphism by definition.\\
We are left to verify that the following diagrams are commutative\small
\begin{center}
\begin{equation}
\label{pentagonocompo}
\begin{tikzpicture}
 \matrix (m) [matrix of math nodes,row sep=2em,column sep=2em,minimum width=2em]
 {
  {\,}  &G F((X\otimes_\C Y) \otimes_\m M) &{\,} \\
G F(X \otimes_\m (Y \otimes_\m M))&{\,}&(X \otimes_\C Y) \otimes_\mathcal{P} GF(M)\\
X \otimes_\mathcal{P} GF(Y \otimes_\m M) & {\,} &X \otimes_\mathcal{P}(Y \otimes_\mathcal{P} GF(M)) \\  };
\path[-stealth]
(m-1-2) edge node [above] {$\scriptstyle{GF(m_{X,Y,M}}) \, \, \, \, \qquad$} (m-2-1)
(m-1-2) edge node [above] {$\scriptstyle{u_{X \otimes_\C Y,M}}$} (m-2-3)
(m-2-1) edge node [left] {$\scriptstyle{u_{X,Y\otimes_\m M}}$} (m-3-1)
(m-3-1) edge node [above] {$\scriptstyle{\text{id}_X \otimes_\mathcal{N} u_{Y,M}}$} (m-3-3)
(m-2-3) edge node [right] {$\scriptstyle{p_{X,Y,GF(M)}}$} (m-3-3);
\end{tikzpicture}
\end{equation}
\end{center}
and
\begin{center}
\begin{equation*}
\begin{tikzpicture}
 \matrix (m) [matrix of math nodes,row sep=2em,column sep=5em,minimum width=2em]
 {
  GF(1 \otimes_\m M) &{\,}& 1 \otimes_\mathcal{P} GF(M)\\
 {\,}&{\,}&{\,}\\
{\,} & F(M) & {\,} \\ };
 \path[-stealth]
 (m-1-1) edge node [above] {$\scriptstyle{u_{1,M}}$} (m-1-3)
(m-1-1) edge node [above] {$\scriptstyle{\qquad \textcolor{white}{F(L_M)}}$} (m-3-2)
(m-1-3) edge node [above] {$\scriptstyle{\textcolor{white}{\tilde{L}_{F(M)}}\qquad}$} (m-3-2);
\end{tikzpicture}
\end{equation*}
\end{center}
for all $X,Y \in \C$ and $M \in \m$. \\
We prove the commutativity of the first diagram. Applying the definition of the natural isomorphism $u$, the considered diagram is the outer rectangle of the following diagram
\begin{center}
\scriptsize
\begin{equation*}
\begin{tikzpicture}
  \matrix (m) [matrix of math nodes,rotate=90,transform shape,row sep=6em,column sep=6em,minimum width=2em]
  {
   GF(X \otimes_\m (Y \otimes_\m M))& GF((X \otimes_\C Y) \otimes_\m M)& G((X \otimes_\C Y) \otimes_\Nn F(M)) \\
     G(X \otimes_\Nn F(Y \otimes_\m M)) &  G(X \otimes_\Nn (Y \otimes_\Nn F(M))) & (X \otimes_\C Y) \otimes_\mathcal{P} GF(M) \\
  X \otimes_\mathcal{P} GF(Y \otimes_\m M) &  X \otimes_\mathcal{P} G (Y \otimes_\Nn F(M))& X \otimes_\mathcal{P} (Y \otimes_\mathcal{P} GF(M)) \\
 };

  \path[-stealth]
    (m-1-2) edge node [above] {$\scriptstyle{GF(m_{X,Y,M})}$} (m-1-1)

    (m-1-1) edge node [right] {\vspace{3cm}\hspace{4cm}\textcolor{red}{{\huge A}}} (m-2-1)
(m-1-2) edge node [above] {$\scriptstyle{G(s_{X \otimes_\C Y, M})}$} (m-1-3)
(m-2-1) edge node [right] {\vspace{2cm}\hspace{3cm}\textcolor{red}{{\huge B}}} (m-3-1)
(m-1-1) edge node [left] {$\scriptstyle{G(s_{X, Y \otimes_\m M})}$} (m-2-1)
(m-2-1) edge node [above] {$\scriptstyle{G(\text{id}_X \otimes_\Nn s_{Y,M})}$} (m-2-2)
(m-1-3) edge node [right] {$\scriptstyle{\, \, \, \, \, \, G(n_{X, Y, F(M)})}$} (m-2-2)
(m-1-3) edge node [right] {$\scriptstyle{t_{X \otimes_\C Y, F(M)}}$} (m-2-3)
(m-2-3) edge node [left] {\vspace{3cm}\hspace{-5cm}\textcolor{red}{{\huge C}}} (m-3-3)
(m-2-1) edge node [left] {$\scriptstyle{t_{X, F(Y \otimes_\m M)}}$} (m-3-1)
(m-3-1) edge node [above] {$\scriptstyle{\text{id}\otimes_\mathcal{P} G(s_{Y,M}) }$} (m-3-2)
(m-2-2) edge node [right] {$\scriptstyle{t_{X, Y \otimes_\Nn F(M)}}$} (m-3-2)
(m-2-3) edge node [right] {$\scriptstyle{p_{X,Y,GF(M)}}$} (m-3-3)
(m-3-2) edge node [above] {$\scriptstyle{\text{id}_X \otimes_\mathcal{P} t_{Y, F(M)}}$} (m-3-3)
  
;
\end{tikzpicture}
\end{equation*}
\end{center}
\normalsize
 for all $X, Y \in \C, M \in \m$. \\

Diagrams \textcolor{red}{A} and \textcolor{red}{C} are commutative since $(F,s)$ and $(G,t)$ are $\C$-module functors. Moreover, the naturality of $t$ implies the commutativity of \textcolor{red}{B}. This implies that the outer rectangle is commutative, i.e. that diagram \eqref{pentagonocompo} commutes. \\
In a similar way, one can prove that diagram \eqref{functor2} is commutative. 
\end{proof}
\subsection{Transport of structure}
\label{transportofstruct}
In this subsection, we show how we can construct new module categories by means of equivalences and transport of structure. Although this procedure is well-known, we spell out it in full details as it is not covered in the literature.\\  \\
Let $(\C, \otimes_\C, a, 1 )$ be a monoidal category. Let $\m$ be a left $\C$-module category, via the action bifunctor $\otimes_\m \colon \C \times \m \to \m$ and with module associativity constraint $m_{X,Y,M}$. Let $F \colon \m \to \Nn$ be an equivalence of categories. Then, there exists a quasi-inverse $G\colon \Nn \to \m$,  i.e. a functor and natural isomorphisms $\eta$ and $\varepsilon$ such that $G \circ F \overset{\varepsilon}{\approx}\text{id}_\m$ and $F \circ G \overset{\eta}{\approx}\text{id}_\Nn$. \\
We can define in a natural way the bifunctor:
\begin{equation}
\begin{split}
\label{leftact}
\otimes_\Nn \colon \C \times \Nn &\to \Nn \\
(- , \sim)&\mapsto F(- \otimes_\m G(\sim)).
\end{split}
\end{equation}
In an analogous way, for every $X, Y \in \C, N \in \Nn$, we can define the maps $$n_{X,Y,N}\colon (X \otimes_\C Y) \otimes_\Nn N \to X \otimes_\Nn (Y \otimes_\Nn N)$$ as follows:
\begin{equation}
\label{assconstraintleft}
n_{X,Y,N}:=F\bigl((\text{id}_X \otimes_\m \varepsilon^{-1}_{Y\otimes_\m G(N)}) \circ m_{X,Y,G(N)}\bigr).
\end{equation}
We are now in a position to prove the following lemma:
\begin{lemma}
\label{left}
Let $(\C, \otimes_\C, a, 1)$ be a monoidal category and let $(\m, \otimes_\m, m)$ be a left $\C$-module category. Let  $F \colon \m \to \Nn$ be an equivalence of categories, with quasi-inverse $G \colon \Nn \to \m$ such that $G \circ F \overset{\varepsilon}{\approx}\text{id}_\m$ and $F \circ G \overset{\eta}{\approx}\text{id}_\Nn$. The category $\Nn$ endowed with $\otimes_\Nn$ as in Equation \eqref{leftact}, and morphisms $n_{X,Y,N}$ for all $X,Y \in \C$, $N \in \Nn$ as in Equation \eqref{assconstraintleft} is a left $\C$-module category.
\end{lemma}
\begin{proof}
By definition, $\otimes_\Nn$ is a bifunctor and the maps $n_{X,Y,N}$ for $X, Y \in \C, \, N \in \Nn$ combine to give natural isomorphisms.\\
Moreover, the functor 
\begin{equation*}
\begin{split}
1 \otimes_\Nn - \colon \Nn& \to \Nn \\
N &\mapsto 1 \otimes_\Nn N = F(1 \otimes_\m G(N)) 
\end{split}
\end{equation*}
 is an autoequivalence, since it is the composition of equivalences. \\
It remains to verify that the pentagon rule holds, i.e. that the following diagram commutes
\begin{center}
\small
\begin{equation}
\begin{tikzpicture}[thick,scale=0.6, every node/.style={scale=0.8}]
  \matrix (m) [matrix of math nodes,row sep=2em,column sep=2em,minimum width=2em]
  {
   {\,}  &((X \otimes_\C Y)\otimes_\C Z) \otimes_\Nn N &{\,} \\
   (X \otimes_\C(Y \otimes_\C Z)) \otimes_\Nn N&{\,}&(X \otimes_\C Y) \otimes_\Nn (Z \otimes_\Nn N)\\
X \otimes_\Nn ((Y \otimes_\C Z)\otimes_\Nn N) & {\,} &X \otimes_\Nn(Y \otimes_\Nn (Z \otimes_\Nn N)) \\  };
  \path[-stealth]
    (m-1-2) edge node [above] {$\scriptstyle{a_{X,Y,Z}\otimes_\Nn \text{id}_N \qquad \qquad}$} (m-2-1)
    (m-1-2) edge node [above] {$\scriptstyle{\qquad n_{X \otimes_\C Y,Z,N}}$} (m-2-3)
    (m-2-1) edge node [left] {$\scriptstyle{n_{X,Y\otimes_\C Z, N}}$} (m-3-1)
(m-3-1) edge node [above] {$\scriptstyle{\text{id}_X \otimes_\Nn n_{Y,Z,N}}$} (m-3-3)
(m-2-3) edge node [right] {$\scriptstyle{n_{X,Y,Z \otimes_\Nn N}}$} (m-3-3);
\end{tikzpicture}
\end{equation}
\end{center}
for every $X,Y,Z \in \C$ and $N \in \Nn$. Using Equation \eqref{leftact} and Equation \eqref{assconstraintleft}, this is equivalent to prove that:
\small
\begin{equation}
\label{pentag}
\begin{split}
&F\bigl((\text{id}_X \otimes_\m GF((\text{id}_Y \otimes_\m \varepsilon^{-1}_{Z \otimes_\m G(N)})\circ m_{Y,Z,G(N)})) \circ (\text{id}_X \otimes_\m \varepsilon^{-1}_{(Y \otimes_\C Z) \otimes_\m G(N)}) \bigr)\circ
\\
&\circ F \bigl( m_{X,Y \otimes_\C Z, G(N)} \circ (a_{X,Y,Z} \otimes_\m \text{id}_{G(N)})\bigr)=\\
&F \hspace{-2pt}\bigl(\hspace{-2pt}(\text{id}_X \otimes_\m \varepsilon^{-1}_{Y \otimes_\m (GF(Z \otimes_\m G(N)))})\hspace{-2pt} \circ\hspace{-2pt}  m_{X,Y,GF(Z \otimes_\m G(N))} \hspace{-2pt}\circ \hspace{-2pt}  (\text{id}_X \otimes_\C \text{id}_Y \otimes_\m \varepsilon^{-1}_{Z \otimes_\m G(N)}) \bigr)\circ \\
&\circ  F( m_{X \otimes_\C Y,Z,G(N)}).
\end{split}
\end{equation}
\normalsize
This is achived using some commutative diagrams. In particular, by the  naturality of $\varepsilon$, the diagrams
\begin{center}
\begin{equation}
\label{A}
\begin{tikzpicture}
  \matrix (m) [matrix of math nodes,row sep=3em,column sep=6em,minimum width=2em]
  {
    GF((Y \otimes_\C Z)\otimes_\m G(N)) & GF(Y \otimes_\m (Z \otimes_\m G(N)))\\
(Y \otimes_\C Z)\otimes_\m G(N)&Y \otimes_\m (Z \otimes_\m G(N)) \\ };
  \path[-stealth]
    (m-1-1) edge node [above] {$\scriptstyle{GF(m_{Y,Z,G(N)})}$} (m-1-2)
    (m-1-1) edge node [left] {$\scriptstyle{\varepsilon_{(Y \otimes_\C Z)\otimes_\m G(N)}}$} (m-2-1)
    (m-2-1) edge node [above] {$\scriptstyle{m_{Y,Z,G(N)}}$} (m-2-2)
(m-1-2) edge node [right] {$\scriptstyle{\varepsilon_{Y \otimes_\m (Z \otimes_\m G(N))}}$} (m-2-2);
\end{tikzpicture}
\end{equation}
\end{center}
and 
\begin{center}
\begin{equation}
\label{C}
\begin{tikzpicture}[thick,scale=0.6, every node/.style={scale=0.95}]
  \matrix (m) [matrix of math nodes,row sep=3em,column sep=6em,minimum width=2em]
  {
    GF(Y \otimes_\m ( Z\otimes_\m G(N)))& Y\otimes_\m (Z \otimes_\m G(N))\\
GF(Y\otimes_\m GF (Z \otimes_\m G(N)))&Y\otimes_\m GF (Z \otimes_\m G(N))\\ };
  \path[-stealth]
    (m-1-1) edge node [above] {$\scriptstyle{\varepsilon_{Y \otimes_\m (Z \otimes_\m G(N))}}$} (m-1-2)
    (m-1-1) edge node [left] {$\scriptstyle{GF(\text{id}_Y \otimes_\m \varepsilon^{-1}_{Z \otimes_\m G(N)})}$} (m-2-1)
    (m-2-1) edge node [above] {$\scriptstyle{\varepsilon_{Y \otimes_\m GF(Z \otimes_\m G(N))}}$} (m-2-2)
(m-1-2) edge node [right] {$\scriptstyle{\text{id}_Y \otimes_\m \varepsilon^{-1}_{Z \otimes_\m G(N)}}$} (m-2-2);
\end{tikzpicture}
\end{equation}
\end{center}
commute, for every $Y,Z \in \C$ and $N \in \Nn$.\\
Furthermore, by the naturality of $m$ with respect to the functors $(-\otimes_\C -) \otimes_\m -$ and $- \otimes_\m (- \otimes_\m -)$, the diagram:
 \begin{center}
 \small
\begin{equation}
\label{B}
\begin{tikzpicture}
  \matrix (m) [matrix of math nodes,row sep=3.5em,column sep=9em,minimum width=2em]
  {
    (X\otimes_\C Y) \otimes_\m GF(Z \otimes_\m G(N))& (X\otimes_\C Y) \otimes_\m (Z \otimes_\m G(N))\\
X\otimes_\m( Y \otimes_\m GF(Z \otimes_\m G(N)))&X\otimes_\m( Y \otimes_\m (Z \otimes_\m G(N))) \\ };
  \path[-stealth]
    (m-1-1) edge node [above] {$\scriptstyle{\text{id}_{X \otimes_\C Y} \otimes_\m \varepsilon_{Z \otimes_\m G(N)}}$} (m-1-2)
    (m-1-1) edge node [left] {$\scriptstyle{m_{X,Y,GF(Z \otimes_\m G(N))}}$} (m-2-1)
    (m-2-1) edge node [above] {$\scriptstyle{\text{id}_X \otimes_\m(\text{id}_Y \otimes_\m \varepsilon_{Z \otimes_\m G(N)})}$} (m-2-2)
(m-1-2) edge node [right] {$\scriptstyle{m_{X,Y,Z \otimes_\m G(N)}}$} (m-2-2);
\end{tikzpicture}
\end{equation}
\end{center}
\normalsize
is commutative for every $X,Y,Z \in \C$ and $ N \in \Nn$.\\
Putting together this series of commutative diagrams with the fact that $\m$ is a $\C$-module category one gets that Equation \eqref{pentag} holds. We refer the reader to the Appendix for full details.
\end{proof}
Moreover, we can prove that the functors $F$ and $G$ as above are $\C$-module functors (Definition \ref{modulefunctor}). 
\begin{lemma}
\label{C-invariants}
Let $(\C, \otimes_\C, a, 1)$ be a monoidal category and let $(\m, \otimes_\m, m)$ be a left $\C$-module category. Let $F \colon \m \to \Nn$ be an equivalence, with quasi-inverse $G \colon \Nn \to \m$ such that $G \circ F \overset{\varepsilon}{\approx}\text{id}_\m$ and $F \circ G \overset{\eta}{\approx}\text{id}_\Nn$ and consider the category $\Nn$ endowed with $\otimes_\Nn$ as in Equation \eqref{leftact}, with the morphism $n_{X,Y,N}$ as in Equation \eqref{assconstraintleft}. \\
Let $s_{X,M}:=F(\text{id}_X \otimes_\m \varepsilon_M^{-1})$ and $t_{X,N}:=\varepsilon_{X \otimes_\m G(N)}$, for every $X \in \C, M \in \m$ and $N \in \Nn$. Then the pairs $(F,s)$ and $(G,t)$ are $\C$-module functors, so $\m$ and $\Nn$ are equivalent as $\C$-module categories.
\end{lemma}
\begin{proof}
We prove this result for $F$, for $G$ the proof will be analogous. \\
Firstly, notice that the morphism $s$ has the required source and target and it is a natural isomorphism, since it is the composition of a natural isomorphism and an equivalence. We have now to check that Diagrams \eqref{functor1} and \eqref{functor2} are commutative.\\
In particular, Diagram \eqref{functor1} reduces to 
\begin{center}
\small
\begin{equation*}
\begin{tikzpicture}
 \matrix (m) [matrix of math nodes,row sep=2em,column sep=1em,minimum width=2em]
 {
  {\,}  & F((X\otimes_\C Y) \otimes_\m M) &{\,} \\
F(X \otimes_\m (Y \otimes_\m M))&{\,}&F((X \otimes_\C Y) \otimes_\m  GF(M))\\
F(X \otimes_\m GF(Y \otimes_\m M)) & {\,} &F(X \otimes_\m GF(Y \otimes_\m GF(M))) \\  };
\path[-stealth]
(m-1-2) edge node [above] {$\scriptstyle{F(m_{X,Y,M})}\qquad$} (m-2-1)
(m-1-2) edge node [above] {$\scriptstyle{s_{X \otimes_\C Y,M}}$} (m-2-3)
(m-2-1) edge node [left] {$\scriptstyle{s_{X,Y\otimes_\m M}}$} (m-3-1)
(m-3-1) edge node [above] {$\scriptstyle{\text{id}_X \otimes_\mathcal{N} s_{Y,M}}$} (m-3-3)
(m-2-3) edge node [right] {$\scriptstyle{n_{X,Y,F(M)}}$} (m-3-3);
\end{tikzpicture}
\end{equation*}
\end{center}
for all $X, Y \in \C$ and $M \in \m$. Using the definition of $n$ and $s$, the previous diagram becomes the outer diagram of the following
\small
\begin{center}
\begin{equation*}
\begin{tikzpicture}
  \matrix (m) [matrix of math nodes,rotate=90,transform shape,row sep=4em,column sep=3em,minimum width=2em]
  {
  F((X \otimes_\C Y) \otimes_\m M)&\,&F((X \otimes_\C Y) \otimes_\m GF(M)) \\
     F(X \otimes_\m (Y \otimes_\m M)) & \, & F(X \otimes_\m (Y \otimes_\m GF(M))) \\
 F(X \otimes_\m GF(Y \otimes_\m M))& \, & F(X \otimes_\m GF(Y \otimes_\m GF(M)))\\
 };

  \path[-stealth]
    (m-1-1) edge node [above] {$\scriptstyle{F(\text{id}_{X \otimes_\C Y} \otimes_\m \varepsilon^{-1}_M)}$} (m-1-3)
    (m-1-1) edge node [left] {$\scriptstyle{F(m_{X,Y,M})}$} (m-2-1)
    (m-2-1) edge node [left] {$\scriptstyle{F(\text{id}_{X} \otimes_\m \varepsilon^{-1}_{Y \otimes_\m M})}$} (m-3-1)
    (m-3-1) edge node [above] {$\scriptstyle{F(\text{id}_X \otimes_\m GF(\text{id}_Y \otimes_\m \varepsilon_M^{-1}))}$} (m-3-3)
    (m-2-1) edge node [above] {$\scriptstyle{F(\text{id}_{X } \otimes_\m (\text{id}_Y \otimes_\m \varepsilon_M^{-1}))}$} (m-2-3)
    (m-1-3) edge node [right] {$\scriptstyle{F(m_{X,Y,GF(M)})}$} (m-2-3)
(m-2-3) edge node [right] {$\scriptstyle{F(\text{id}_X \otimes_\m \varepsilon^{-1}_{Y \otimes_\m GF(M)})}$} (m-3-3)
;
\end{tikzpicture}
\end{equation*}
\end{center}
for all $X, Y \in \C$ and $M \in \m$. By naturality of $m$, the upper rectangle commutes, while the naturality of $\varepsilon$ makes the lower rectangle commutative. As a consequence, we get the commutativity of the outer rectangle.\\
We are left to prove that Diagram \eqref{functor2} is commutative, that is the following
\small
\begin{center}
\begin{equation*}
\begin{tikzpicture}
  \matrix (m) [matrix of math nodes,rotate=90,transform shape,row sep=6em,column sep=5em,minimum width=2em]
  {
  F(1 \otimes_\m M)&F(1 \otimes_\m GF(M)) \\
    F(M) & FGF(M)\\
 };

  \path[-stealth]
   (m-1-1) edge node [above] {$\scriptstyle{F(\text{id}_1 \otimes_\m \varepsilon^{-1}_M)}$} (m-1-2)
    (m-1-1) edge node [left] {$\scriptstyle{\,}$} (m-2-1)
    (m-2-1) edge node [above] {$\scriptstyle{\varepsilon^{-1}_M}$} (m-2-2)
    (m-1-2) edge node [right] {$\scriptstyle{\,}$} (m-2-2)
;
\end{tikzpicture}
\end{equation*}
\end{center}
is commutative for any $M \in \m$. This follows from the naturality of $\varepsilon$.
\end{proof}
We can dualize the above construction for a right $\C$-module category. \\ \\
Let $\C$ be a monoidal category and let $\m$ be a right $\C$-module category, via the action bifunctor $\otimes^\m \colon \m \times \C \to \m$ and with module associativity constraint $m_{X,Y,M}^r$. Consider as before, an equivalence $F \colon \m \to \Nn$ with quasi-inverse $G$ and natural isomorphisms $\eta$ and $\varepsilon$, such that $G \circ F \overset{\varepsilon}{\approx}\text{id}_\m$ and $F \circ G \overset{\eta}{\approx}\text{id}_\Nn$. Then, we have the right-handed version of the definitions given above.
We define the bifunctor 
\begin{equation}
\label{rightact}
\begin{split}
\otimes^\Nn \colon \Nn \times \C &\to \Nn \\
(-, \sim) &\mapsto F(G(-) \otimes^\m \sim)
\end{split}
\end{equation}
and the maps $$n_{N,X,Y}^r \colon (N \otimes^\Nn X) \otimes^\Nn Y \to N \otimes^\Nn (X \otimes Y)$$ as 
\begin{equation}
\label{rightass}
n_{N,X,Y}^r:=F(m_{G(N), X, Y}^r\circ(\varepsilon_{G(N) \otimes^\m X}\otimes^\m \text{id}_Y))
\end{equation}
for every $X,Y \in \C$ and $N \in \Nn$.\\
In the same way as for the left $\C$-module category, we can prove the right-handed version of Lemma \ref{left} and Lemma \ref{C-invariants}.
\begin{lemma}
\label{right}
Let $(\C, \otimes_\C, a, 1)$ be a monoidal category and let $(\m, \otimes^\m,m^r)$ be a right $\C$-module category.  Let $F \colon \m \to \Nn$ be an equivalence of categories, with quasi-inverse $G \colon \Nn \to \m$  and natural isomorphisms $\eta$ and $\varepsilon$ such that $G \circ F \overset{\varepsilon}{\approx}\text{id}_\m$ and $F \circ G \overset{\eta}{\approx}\text{id}_\Nn$. The category $\Nn$ endowed with $\otimes^\Nn$ as in Equation \eqref{rightact}, with the morphism $n^r_{N, X, Y}$ as in Equation \eqref{rightass} is a right $\C$-module category.\\
Moreover, let $s^r_{X,M}:=F(\varepsilon_M^{-1} \otimes^\m \text{id}_X)$ and let $t_{X,N}^r:=\varepsilon_{G(N) \otimes^\m X}$, for all $X \in \C, M \in \m$ and $N \in \Nn$. Then, the pairs $(F,s^r)$ and $(G,t^r)$ are right $\C$-module functors.
\end{lemma}
It is now natural to expect that the transported structure of a $(\C, \D)$-bimodule category by means of an equivalence is again a $(\C, \D)$-bimodule category.\\
Let $(\C, \otimes_\C, a, 1)$ and $(\D,\otimes_\mathcal{D},\tilde{a}, \tilde{1})$ be two monoidal categories and let \\$(\m, m, m^r, b)$ be a $(\C,\D)$-bimodule category. Consider $F \colon \m \to \Nn$ an equivalence, with quasi-inverse $G$, and natural isomorphisms $\eta$ and $\varepsilon$ such that $G \circ F \overset{\varepsilon}{\approx}\text{id}_\m$ and $F \circ G \overset{\eta}{\approx}\text{id}_\Nn$. Using the above notation we define for every $X \in \C, Z \in \D$ and $N \in \mathcal{N}$, the maps $$p_{X,N,Z}\colon (X \otimes_\Nn N) \otimes^\Nn Z \to X \otimes_\Nn (N \otimes^\Nn Z)$$ as
\begin{equation}
\label{compatible}
p_{X,N,Z}:=F\bigl( (\text{id}_X \otimes_\m \varepsilon^{-1}_{G(N) \otimes^\m Z}) \circ b_{X,G(N),Z} \circ (\varepsilon_{X \otimes_\m G(N)} \otimes^\m \text{id}_Z)\bigr).
\end{equation}

We are now in a position to prove the following result.
\begin{lemma}
\label{lemmabimod}
Let $(\C, \otimes_\C, a, 1)$ and $(\D, \otimes_\D, \tilde{a}, \tilde{1})$ be monoidal categories and let $(\m, \otimes_\m, \otimes^\m, m, m^r, b)$ be a $(\C,\D)$-bimodule category. Consider $F \colon \m \to \Nn$ an equivalence, with quasi-inverse $G$, and natural isomorphisms $\eta$ and $\varepsilon$ such that $G \circ F \overset{\varepsilon}{\approx}\text{id}_\m$ and $F \circ G \overset{\eta}{\approx}\text{id}_\Nn$. The category $\Nn$ endowed with $\otimes_\Nn$, $\otimes^\Nn$ as in Equation \eqref{leftact} and Equation \eqref{rightact}, the morphisms $n_{X,Y,N}$,  $n_{N, X, Y}^r$ and $p_{X,N,Z}$ for any $X, Y, Z \in \C$ and $N \in \Nn$ as in Equation \eqref{assconstraintleft}, \eqref{rightass} and in Equation \eqref{compatible}, respectively, is a  $(\C,\D)$-bimodule category. \\
Furthermore, the functors $(F, s, s^r)$ and $(G, t, t^r)$ with $s$ and $t$ as in Lemma \ref{C-invariants} and $s^r$, $t^r$ as in Lemma \ref{right} are $(\C, \mathcal{D})$-bimodule functors.
\end{lemma}
\begin{proof}
We first prove that $(\Nn, \otimes_\Nn, \otimes^\Nn, n, n^r, p)$ is a $(\C, \mathcal{D})$-bimodule category.\\
By Lemma \ref{left}, Lemma \ref{right} and by the fact that $p$ is a natural isomorphism by definition, we only need to prove that diagrams \eqref{bim1} and \eqref{bim2} commute.\\
Consider the first diagram. We want to show that the diagram
\begin{center}
\begin{equation*}
\begin{tikzpicture}[thick,scale=0.6, every node/.style={scale=0.85}]
  \matrix (m) [matrix of math nodes,row sep=2em,column sep=2em,minimum width=2em]
  {
   {\,}  &((X \otimes_\C Y)\otimes_\Nn N) \otimes^\Nn Z &{\,} \\
   (X \otimes_\Nn (Y \otimes_\Nn N)) \otimes^\Nn Z&{\,}&(X \otimes_\C Y) \otimes^\Nn (N \otimes^\Nn Z)\\
X \otimes_\Nn ((Y \otimes_\Nn N)\otimes^\Nn Z) & {\,} &X \otimes_\Nn (Y \otimes_\Nn (N \otimes^\Nn Z)) \\  };
  \path[-stealth]
    (m-1-2) edge node [above] {$\scriptstyle{n_{X,Y,N}\otimes^\Nn \text{id}_Z \qquad \qquad}$} (m-2-1)
    (m-1-2) edge node [above] {$\scriptstyle{\qquad p_{X \otimes_\C Y,N,Z}}$} (m-2-3)
    (m-2-1) edge node [left] {$\scriptstyle{p_{X,Y\otimes_\Nn N, Z}}$} (m-3-1)
(m-3-1) edge node [above] {$\scriptstyle{\text{id}_X \otimes_\Nn p_{Y,N,Z}}$} (m-3-3)
(m-2-3) edge node [right] {$\scriptstyle{n_{X,Y,N \otimes^\Nn Z}}$} (m-3-3);
\end{tikzpicture}
\end{equation*}
\end{center}
is commutative for every $X,Y \in \C, Z \in \D$ and $N \in \Nn$.
By Equations \eqref{assconstraintleft} and \eqref{compatible}, this is equivalent to verify that 
\small
\begin{equation}
\label{bimod}
\begin{split}
&F(\text{id}_X \otimes_\m GF((\text{id}_Y \otimes_\m \varepsilon^{-1}_{G(N) \otimes^\m Z}) \circ b_{Y,G(N),Z} \circ(\varepsilon_{Y \otimes_\m G(N)} \otimes^\m \text{id}_Z))) \circ \\
&F((\text{id}_X \otimes_\m \varepsilon^{-1}_{GF(Y \otimes_\m G(N))\otimes^\m \text{id}_Z}) \hspace{-2pt}\circ \hspace{-2pt}b_{X,GF(Y \otimes_\m G(N)),Z}\hspace{-2pt} \circ \hspace{-2pt} (\varepsilon_{X \otimes_\m GF(Y \otimes_\m G(N))} \otimes^\m \text{id}_Z)) \circ \\
&F(GF((\text{id}_X \otimes_\m \varepsilon^{-1}_{Y \otimes_\m G(N)}) \circ m_{X,Y,G(N)}) \otimes^\m \text{id}_Z)=\\
&F((\text{id}_X \otimes_\m \varepsilon^{-1}_{Y \otimes_\m GF(G(N) \otimes^\m Z)}) \circ m_{X,Y,GF(G(N) \otimes^\m Z)} \circ (id_{X \otimes_\C Y} \otimes_\m \varepsilon^{-1}_{G(N) \otimes^\m Z})) \circ \\
&F(b_{X \otimes_\C Y,G(N),Z} \circ (\varepsilon_{(X \otimes_\C Y) \otimes_\m G(N)} \otimes^\m \text{id}_Z)).
\end{split}
\end{equation}
\normalsize
We will make use of a series of commutative diagrams.
By naturality of $\varepsilon$, the following five diagrams
\begin{center}
\begin{equation}
\label{1}
\begin{tikzpicture}
  \matrix (m) [matrix of math nodes,row sep=3em,column sep=6em,minimum width=2em]
  {
    GF(X \otimes_\m (Y \otimes_\m G(N))) & X \otimes_\m (Y \otimes_\m G(N))\\
GF((X \otimes_\C Y) \otimes_\m G(N)))&(X \otimes_\C Y) \otimes_\m G(N)), \\ };
  \path[-stealth]
    (m-1-1) edge node [above] {$\scriptstyle{\varepsilon_{X \otimes_\m (Y \otimes_\m G(N))}}$} (m-1-2)
    (m-2-1) edge node [left] {$\scriptstyle{GF(m_{X,Y,G(N)})}$} (m-1-1)
    (m-2-1) edge node [above] {$\scriptstyle{\varepsilon_{(X \otimes_\C Y) \otimes_\m G(N)}}$} (m-2-2)
(m-2-2) edge node [right] {$\scriptstyle{m_{X,Y,G(N)}}$} (m-1-2);
\end{tikzpicture}
\end{equation}
\end{center}

\begin{center}
\small
\begin{equation}
\label{2}
\begin{tikzpicture}
  \matrix (m) [matrix of math nodes,row sep=3em,column sep=9em,minimum width=2em]
  {
    GF(X \otimes_\m GF(Y \otimes_\m G(N))) & X \otimes_\m GF(Y \otimes_\m G(N))\\
  GF(X \otimes_\m (Y \otimes_\m G(N)))&X \otimes_\m (Y \otimes_\m G(N)), \\ };
  \path[-stealth]
    (m-1-1) edge node [above] {$\scriptstyle{\varepsilon_{X \otimes_\m GF(Y \otimes_\m G(N))}}$} (m-1-2)
    (m-1-1) edge node [left] {$\scriptstyle{GF(\text{id}_X \otimes_\m \varepsilon_{Y \otimes_\m G(N)})}$} (m-2-1)
    (m-2-1) edge node [above] {$\scriptstyle{\varepsilon_{X \otimes_\m (Y \otimes_\m G(N))}}$} (m-2-2)
(m-1-2) edge node [right] {$\scriptstyle{\text{id}_X \otimes_\m \varepsilon_{Y \otimes_\m G(N)}}$} (m-2-2);
\end{tikzpicture}
\end{equation}
\end{center}

\begin{center}
\begin{equation}
\label{3}
\begin{tikzpicture}
  \matrix (m) [matrix of math nodes,row sep=3em,column sep=6em,minimum width=2em]
  {
    GF((Y \otimes_\m G(N)) \otimes^\m Z) & (Y \otimes_\m G(N)) \otimes^\m Z\\
  GF(Y \otimes_\m (G(N) \otimes^\m Z))&Y \otimes_\m (G(N) \otimes^\m Z), \\ };
  \path[-stealth]
    (m-1-1) edge node [above] {$\scriptstyle{\varepsilon_{(Y \otimes_\m G(N)) \otimes^\m Z}}$} (m-1-2)
    (m-1-1) edge node [left] {$\scriptstyle{GF(b_{Y,G(N),Z})}$} (m-2-1)
    (m-2-1) edge node [above] {$\scriptstyle{\varepsilon_{Y \otimes_\m (G(N) \otimes^\m Z)}}$} (m-2-2)
(m-1-2) edge node [right] {$\scriptstyle{b_{Y,G(N),Z}}$} (m-2-2);
\end{tikzpicture}
\end{equation}
\end{center}
\begin{center}
\begin{equation}
\label{4}
\begin{tikzpicture}[thick,scale=0.6, every node/.style={scale=0.9}]
  \matrix (m) [matrix of math nodes,row sep=3em,column sep=6em,minimum width=2em]
  {
    GF(GF(Y \otimes_\m G(N)) \otimes^\m Z) &GF(Y \otimes_\m G(N)) \otimes^\m Z\\
  GF((Y \otimes_\m G(N)) \otimes^\m Z)&(Y \otimes_\m G(N)) \otimes^\m Z,\\ };
  \path[-stealth]
    (m-1-1) edge node [above] {$\scriptstyle{\varepsilon_{GF(Y \otimes_\m G(N)) \otimes^\m Z}}$} (m-1-2)
    (m-1-1) edge node [left] {$\scriptstyle{GF(\varepsilon_{Y \otimes_\m G(N)} \otimes^\m \text{id}_Z)}$} (m-2-1)
    (m-2-1) edge node [above] {$\scriptstyle{\varepsilon_{(Y \otimes_\m G(N)) \otimes^\m Z}}$} (m-2-2)
(m-1-2) edge node [right] {$\scriptstyle{\varepsilon_{Y \otimes_\m G(N)} \otimes^\m \text{id}_Z}$} (m-2-2);
\end{tikzpicture}
\end{equation}
\end{center}
and
\begin{center}
\begin{equation}
\label{5}
\begin{tikzpicture}[thick,scale=0.6, every node/.style={scale=0.9}]
  \matrix (m) [matrix of math nodes,row sep=3em,column sep=6em,minimum width=2em]
  {
    GF(Y \otimes_\m (G(N) \otimes^\m Z)) &Y \otimes_\m (G(N) \otimes^\m Z)\\
  GF(Y \otimes_\m GF(G(N) \otimes^\m Z))&Y \otimes_\m GF(G(N) \otimes^\m Z)\\ };
  \path[-stealth]
    (m-1-1) edge node [above] {$\scriptstyle{\varepsilon_{Y \otimes_\m (G(N) \otimes^\m Z)}}$} (m-1-2)
    (m-1-1) edge node [left] {$\scriptstyle{GF(\text{id}_Y \otimes_\m \varepsilon^{-1}_{G(N) \otimes^\m Z})}$} (m-2-1)
    (m-2-1) edge node [above] {$\scriptstyle{\varepsilon_{Y \otimes_\m GF(G(N) \otimes^\m Z)}}$} (m-2-2)
(m-1-2) edge node [right] {$\scriptstyle{\text{id}_Y \otimes_\m \varepsilon^{-1}_{G(N) \otimes^\m Z}}$} (m-2-2);
\end{tikzpicture}
\end{equation}
\end{center}
commute for every $N \in \Nn$, $X,Y \in \C$ and $Z \in \D$.\\
Moreover, by the naturality of $m$ the following diagram:
\begin{center}
\begin{equation}
\label{6}
\begin{tikzpicture}[thick,scale=0.6, every node/.style={scale=0.9}]
  \matrix (m) [matrix of math nodes,row sep=3em,column sep=6em,minimum width=2em]
  {
    (X \otimes_\C Y) \otimes_\m GF(G(N) \otimes^\m Z) & X \otimes_\m (Y \otimes_\m GF(G(N) \otimes^\m Z)) \\
  (X \otimes_\C Y) \otimes_\m (G(N) \otimes^\m Z)& X \otimes_\m (Y \otimes_\m (G(N) \otimes^\m Z))\\ };
  \path[-stealth]
    (m-1-1) edge node [above] {$\scriptstyle{m_{X,Y, GF(G(N) \otimes^\m Z)}}$} (m-1-2)
    (m-1-1) edge node [left] {$\scriptstyle{\text{id}_{X \otimes_\C Y} \otimes_\m \varepsilon_{G(N) \otimes^\m Z}}$} (m-2-1)
    (m-2-1) edge node [above] {$\scriptstyle{m_{X,Y,G(N) \otimes^\m Z}}$} (m-2-2)
(m-1-2) edge node [right] {$\scriptstyle{\text{id}_ X \otimes_\m (\text{id}_Y \otimes_\m \varepsilon_{G(N) \otimes^\m Z})}$} (m-2-2);
\end{tikzpicture}
\end{equation}
\end{center}
commutes for every $X,Y \in \C, Z \in \D$ and $N \in \Nn$.\\
Furthermore, by naturality of $b$ the diagram
\begin{center}
\small
\begin{equation}
\label{7}
\begin{tikzpicture}[thick,scale=0.6, every node/.style={scale=0.85}]
  \matrix (m) [matrix of math nodes,row sep=3em,column sep=6em,minimum width=2em]
  {
    (X \otimes_\m GF(Y \otimes_\m G(N))) \otimes^\m Z & X \otimes_\m (GF(Y \otimes_\m G(N)) \otimes^\m Z) \\
  (X \otimes_\m (Y \otimes_\m G(N))) \otimes^\m Z& X \otimes_\m ((Y \otimes_\m G(N)) \otimes^\m Z)\\ };
  \path[-stealth]
    (m-1-1) edge node [above] {$\scriptstyle{b_{X, GF(Y \otimes_\m G(N)),Z}}$} (m-1-2)
    (m-1-1) edge node [left] {$\scriptstyle{(\text{id}_{X} \otimes_\m \varepsilon_{Y \otimes_\m G(N)}) \otimes^\m \text{id}_Z}$} (m-2-1)
    (m-2-1) edge node [above] {$\scriptstyle{b_{X, Y \otimes_\m G(N),Z}}$} (m-2-2)
(m-1-2) edge node [right] {$\scriptstyle{\text{id}_{X} \otimes_\m (\varepsilon_{Y \otimes_\m G(N)} \otimes^\m \text{id}_Z)}$} (m-2-2);
\end{tikzpicture}
\end{equation}
\end{center}
is commutative for every $X,Y \in \C, Z \in \D$ and $N \in \Nn$.\\
The commutativity of the above diagrams together with the fact that $\m$ is $(\C, \D)$-bimodule category implies that Equation \eqref{bimod} holds. Further details can be found in the Appendix \ref{appendix}.\\
The other diagram is completely analogous to prove. \\ \\
We now show that $(F, s, s^r)$ is a $(\C, \mathcal{D})$-bimodule functor. By Lemma \ref{C-invariants}  and Lemma \ref{right} the pairs $(F,s)$ and $(F,s^r)$ are respectively a left $\C$-module functor and a right $\mathcal{D}$-module functor. We are left to prove the commutativity of Diagram \eqref{functorbimodule}.  By the definition of $s$, $s^r$ and $p$ this is equivalent to verify that the following diagram commutes
\begin{center}
\small{
\begin{equation}
\label{bimodulefunctor2}
\begin{tikzpicture}
  \matrix (m) [matrix of math nodes,rotate=90,transform shape,row sep=4em,column sep=0.5em,minimum width=1em]
  {
  F((X \otimes_\m M) \otimes^\m Y)&\,& \, &F(X \otimes_\m (M  \otimes^\m Y)) \\
    F(GF(X \otimes_\m M) \otimes^\m Y) & \, & \, & F(X \otimes_\m  GF(M  \otimes^\m Y))\\
 F(GF(X \otimes_\m GF(M)) \otimes^\m Y) & \, & \, &  F(X \otimes_\m  GF(GF(M)  \otimes^\m Y))\\
F((X \otimes_\m GF(M)) \otimes^\m Y)& \, & \, &  F(X \otimes_\m  (GF(M)  \otimes^\m Y))\\
 };

  \path[-stealth]
    (m-1-1) edge node [above] {$\scriptstyle{F(b_{X, M, Y})}$} (m-1-4)
    (m-1-1) edge node [left] {$\scriptstyle{F(\varepsilon^{-1}_{X \otimes_\m M} \otimes^\m \text{id}_Y)}$} (m-2-1)
    (m-2-1) edge node [left] {$\scriptstyle{F(GF(\id_X \otimes_\m \varepsilon^{-1}_M) \otimes^\m \id_Y)}$} (m-3-1)
    (m-3-1) edge node [left] {$\scriptstyle{F(\varepsilon_{X \otimes_\m GF(M)} \otimes^\m \id_Y)}$} (m-4-1)
    (m-4-1) edge node [above] {$\scriptstyle{F(b_{X, GF(M), Y})}$} (m-4-4)
(m-1-4) edge node [right] {$\scriptstyle{F(\id_X \otimes_\m \varepsilon^{-1}_{M \otimes^\m Y})}$} (m-2-4)
(m-2-4) edge node [right] {$\scriptstyle{F(\id_X \otimes_\m GF(\varepsilon^{-1}_M \otimes^\m \id_Y))}$} (m-3-4)
(m-4-4) edge node [right] {$\scriptstyle{F(\id_X \otimes_\m \varepsilon^{-1}_{GF(M) \otimes^\m Y})}$} (m-3-4)
;
\end{tikzpicture}
\end{equation}}
\end{center}
for every $X \in \C$, $Y \in \mathcal{D}$ and $M \in \m$.\\
For this purpose, we will use the following diagrams
\begin{equation}
\label{nat1}
\begin{tikzpicture}[thick,scale=0.6, every node/.style={scale=0.9}]
  \matrix (m) [matrix of math nodes,row sep=3em,column sep=8em,minimum width=2em]
  {
    F((X \otimes_\m M)\otimes^\m Y) &F((X \otimes_\m GF(M))\otimes^\m Y)\\
  F(GF(X \otimes_\m M) \otimes^\m Y)& F(GF(X \otimes_\m GF(M)) \otimes^\m Y)\\ };
  \path[-stealth]
    (m-1-1) edge node [above] {$\scriptstyle{F((\id_X \otimes_\m \varepsilon^{-1}_M) \otimes^\m \id_Y})$} (m-1-2)
    (m-1-1) edge node [left] {$\scriptstyle{F(\varepsilon^{-1}_{X \otimes_\m M} \otimes^\m \id_Y)}$} (m-2-1)
    (m-2-1) edge node [above] {$\scriptstyle{F(GF(\id_X \otimes_\m \varepsilon^{-1}_M) \otimes^\m \id_Y)}$} (m-2-2)
(m-1-2) edge node [right] {$\scriptstyle{F(\varepsilon^{-1}_{X \otimes_\m GF(M)} \otimes^\m Y)}$} (m-2-2);
\end{tikzpicture}
\end{equation}
and
\small
\begin{equation}
\label{nat2}
\begin{tikzpicture}[thick,scale=0.6, every node/.style={scale=0.9}]
  \matrix (m) [matrix of math nodes,row sep=3em,column sep=7em,minimum width=2em]
  {
   F(X \otimes_\m (M \otimes^\m Y)) & F(X \otimes_\m GF(M \otimes^\m Y))\\
  F(X \otimes_\m (GF(M) \otimes^\m Y))& F(X \otimes_\m GF(GF(M) \otimes^\m Y))\\ };
  \path[-stealth]
    (m-1-1) edge node [above] {$\scriptstyle{F(\id_X \otimes_\m \varepsilon^{-1}_{M \otimes^\m Y})}$} (m-1-2)
    (m-1-1) edge node [left] {$\scriptstyle{F(\id_X \otimes_\m (\varepsilon^{-1}_M \otimes^\m \id_Y))}$} (m-2-1)
    (m-2-1) edge node [above] {$\scriptstyle{F(\id_X \otimes_\m \varepsilon^{-1}_{GF(M) \otimes^\m Y})}$} (m-2-2)
(m-1-2) edge node [right] {$\scriptstyle{F(\id_X \otimes_\m GF(\varepsilon^{-1}_M \otimes^\m \id_Y))}$} (m-2-2);
\end{tikzpicture}
\end{equation}
which commute for every $X \in \C$, $M \in \m$ and $Y \in \mathcal{D}$ in virtue of the naturality of $\varepsilon$. Thus diagram \eqref{bimodulefunctor2} reads as
\begin{equation*}
\begin{tikzpicture}[thick,scale=0.6, every node/.style={scale=0.9}]
  \matrix (m) [matrix of math nodes,row sep=3em,column sep=6em,minimum width=2em]
  {
    F((X \otimes_\m M) \otimes^\m Y) & F(X \otimes_\m (M \otimes^\m Y))\\
 F((X \otimes_\m GF(M)) \otimes^\m Y)& F(X \otimes_\m (GF(M) \otimes^\m Y))\\ };
  \path[-stealth]
    (m-1-1) edge node [above] {$\scriptstyle{F(b_{X, M, Y})}$} (m-1-2)
    (m-1-1) edge node [left] {$\scriptstyle{F((\id_X \otimes_\m \varepsilon_M^{-1}) \otimes^\m \id_Y)}$} (m-2-1)
    (m-2-1) edge node [above] {$\scriptstyle{F(b_{X, GF(M), Y})}$} (m-2-2)
(m-1-2) edge node [right] {$\scriptstyle{F(\id_X \otimes_\m (\varepsilon^{-1}_M \otimes^\m \id_Y))}$} (m-2-2);
\end{tikzpicture}
\end{equation*}
for every $X \in \C$, $M \in \m$ and $Y \in \mathcal{D}$ and it is commutative by the naturality of $b$.\\
The proof for $(G,t,t^r)$ is analogous.
\end{proof}
\normalsize
\subsection{Finite $W$-algebras}
\label{preliminW}
In this subsection we recall the construction of a finite-dimensional $W$-algebra. Moreover, we state an important equivalence between categories of modules due to Skryabin.\\ \\
Let $\g$ be a finite-dimensional reductive Lie algebra over $\mathbb{C}$ and let $e \in \g$ be nilpotent. 
A $\Z$-grading
$$
\g=\bigoplus_{j \in \Z}\g(j)
$$
of $\g$ is called a \emph{good grading} for $e$ if
\begin{enumerate}
\item $e \in \g(2)$;
\item $\g^e \subseteq \bigoplus_{j \geq 0}\g(j)$, where $\g^e$ stands for the centralizer of $e$ in $\g$;
\item $\mathfrak{z}(\g) \subseteq \g(0)$, where $\mathfrak{z}(\g)$ denotes the centre of $\g$.
\end{enumerate}
Let $e$ be a nilpotent element in $\g$. By Jacobson-Morozov Theorem, there is  an $\mathfrak{sl}_2$-triple $(e,h,f)$, associated to $e$. \\
The standard good grading is the one induced by the adjoint action of $h$, called the Dynkin grading, that is
$$\g(i)=\{ x \in \g \, | \, [h,x]=ix \}.$$
\begin{remark}
\label{semisimplegrading}
Every grading of $\g$ satisfying $\mathfrak{z}(\g) \subseteq \g(0)$ is induced by the adjoint action of a semisimple element (\cite[Proposition 20.1.5]{TY}).
\end{remark}
Throughout this subsection, we fix a $\mathbb{Z}$-good grading. \\
Consider $(\cdot|\cdot)$ a non-degenerate symmetric invariant bilinear form on $\g$. Define $\chi \in \g^*$ in the following way:
$$\chi \colon \g \to \mathbb{C},  \qquad x \mapsto (e|x).$$  

By definition of good grading, $e \in \g(2)$ and $\chi(x)=(e|x)=0$ for every $x \in \g(j)$, unless $j =-2$.\\
Let $\langle \cdot | \cdot \rangle$  be the non-degenerate alternating bilinear form on $\g(-1)$ defined by
$$\langle x|y \rangle:=\chi([y,x]).
$$
Fix $\ell$ an isotropic subspace of $\g(-1)$ with respect to $\langle \cdot | \cdot \rangle $. \\Let $\ell^{\perp}=\{ x \in \g(-1) \, |\, \langle x, y \rangle =0 \, \, \text{for all } \, y \in \ell \}$, so $\ell \subseteq \ell^{\perp}$.
We define the following subalgebras:
$$\mathfrak{m}_\ell = \ell \oplus \bigoplus_{j < -1} \g(j), \qquad  \mathfrak{n}_\ell = \ell^{\perp} \oplus \bigoplus_{j < -1} \g(j),
$$
so $\mathfrak{m}_\ell \subseteq \mathfrak{n}_\ell$. The algebras $\mathfrak{n}_\ell$ and $\mathfrak{m}_\ell$ are nilpotent, because they are subalgebras of the nilpotent algebra $\bigoplus_{j<0}\g(j)$. Moreover, $\mathfrak{m}_\ell $ is a Lie ideal of $\mathfrak{n}_\ell$.\\\\
Notice that the map $\chi$ restricts to a character of $\mathfrak{m}_\ell$. 
Thus, we can define $$Q_\ell:= U(\g) \otimes_{U(\mathfrak{m}_\ell)}\mathbb{C}_\chi,$$
where $\mathbb{C}_\chi$ is the 1-dimensional left $U(\mathfrak{m}_\ell)$-module obtained from the character $\chi$. In particular, $Q_\ell \simeq U(\g)/I_\ell$, where $I_\ell$ is the left ideal generated by $x-\chi(x)$ for every $x \in \mathfrak{m}_\ell$. The left multiplication in $U(\g)$ induces an action on $Q_\ell$. Specifically,  $$x.(y+I_\ell)=xy+I_\ell,$$ for every $x, y \in U(\g)$. In addition, there is an induced ad$(\mathfrak{n}_\ell)$-action on $Q_\ell$, since the ideal $I_\ell$ is stable under the action of $\mathfrak{n}_\ell$.  \\
Hence, it makes sense to define $$H_\ell:=Q_\ell^{\text{ad} \, \mathfrak{n}_\ell},$$ that is the subspace of all $x+ I_\ell$, with $x \in U(\g)$ such that 
\begin{equation*}
  \tag{$\clubsuit$}
yx-xy \in I_\ell \qquad \text{for all} \qquad y \in \mathfrak{n}_\ell.
  \label{walgebra}
\end{equation*}
We can define an algebra structure on $H_\ell$ via $$(x+I_\ell)(y+ I_\ell)=xy + I_\ell,$$ for $x+I_\ell, y+ I_\ell \in H_\ell$ and the multiplication is well defined (see \cite[Section 1]{GG}). \\
The algebra $H_\ell$ is called the \emph{finite W-algebra} associated with $e$.
\begin{remark}
We stress the fact that the finite $W$-algebra $H_\ell$ does not depend on the choice of the good grading (\cite[Theorem 1]{BG}) and neither from the Lagrangian (\cite[Theorem 4.1]{GG}), but it depends only on the adjoint orbit of $e$.
\end{remark}
\begin{remark}
\label{Qbimodule}
Observe that $Q_\ell$ is a $U(\g)$-$H_\ell$ -bimodule, where $U(\g)$ acts on $Q_\ell$ by left multiplication and the right action of $H_\ell$ on $Q_\ell$ is induced by right multiplication. 
\end{remark}
In the following we require that $\ell$ is a Lagrangian subspace of $\g(-1)$,  that is $\ell=\ell^{\perp}$ so that $\mathfrak{n}_\ell=\mathfrak{m}_\ell$.\\
Let $e$, $\mathfrak{n}_\ell=\mathfrak{m}_\ell$ and $\chi$ be as above.\\ \\
From now on $\m$ will denote the category of $U(\g)$-modules on which $x-\chi(x)$ acts locally nilpotently for each $x \in \mathfrak{m}_\ell$. It is called the category of Whittaker modules.\\Also, from now on $\Nn$ will denote the category of finitely-generated $H_\ell$-modules.\\
When the good grading for $\g$ is the Dynkin one, an equivalence between $\m$ and $\Nn$ was described by Skryabin in \cite[Appendix, Theorem 1]{P}. Moreover, under the same assumption, Gan and Ginzburg gave an alternative proof of Skryabin equivalence in \cite[Theorem 6.1]{GG}.\\
Goodwin in \cite[Theorem 3.14]{G} showed that the same equivalence holds also when the grading of $\g$ is a general good grading. Before exhibiting this equivalence, we put for $M \in \m$
$$ \text{Wh}(M):=\{ v \in M \, | \, x.v=\chi(x)v \, \, \, \text{for all}\, x \in \mathfrak{m}_\ell\}.$$
\begin{remark}
For $M \in \m$, the subspace $\text{Wh}(M)$ is an $H_\ell$-module via the action:
$$
(x+I_\ell). m=x.m,
$$
for every  $m \in \text{Wh}(M)$, $x \in U(\g)$ satisfying \eqref{walgebra}. 
\end{remark}
Then, with the usual restriction of morphisms, $\text{Wh}$ defines a functor from the category of Whittaker modules $\m$ to the category of finitely generated $H_\ell$ -modules.\\
By Remark \ref{Qbimodule}, we also have a functor 
$$ Q_\ell \otimes_{H_\ell}- \colon \Nn \to U(\g)\text{-mod},$$
and in particular $Q_\ell \otimes_{H_\ell} N$ with $U(\g)$-action by left multiplication is a Whittaker module for every $N \in \Nn$. \\
We are now in a position to state the following theorem. 
\begin{theorem}\cite[Theorem 3.14]{G}
\label{equi}
Let $M \in \m$ and $N \in \Nn$.
The functors $M \mapsto\text{Wh}(M)$ and $N \mapsto Q_\ell \otimes_{H_\ell}N$ are quasi inverse equivalences between the category of Whittaker modules $\m$ and the category of finitely-generated $H_\ell$-modules $\Nn$.
\end{theorem} 
\begin{remark}
\label{epsilonremark}
The natural isomorphism $\varepsilon \colon Q_\ell \otimes_{H_\ell} \text{Wh} \to \text{id}$ is given by
\begin{equation}
\begin{split}
\label{epsilon}
\varepsilon_M \colon Q_\ell \otimes_{H_\ell} \text{Wh}(M) &\to M\\
(u+I_\ell) \otimes m &\mapsto u.m
\end{split}
\end{equation}
for $M \in \m$ (\cite[Theorem 3.14]{G}).
\end{remark}
\subsection{Finite $W$-algebra by stages} \label{preliminariestage}
This subsection aims at exhibiting the  Skryabin equivalence by stages, i.e. the functors $\text{Wh}_0$ and $Q_0 \otimes_{H_0}$ introduced in \cite{GJ}.\\ \par
Let $\g$ be a finite-dimensional simple Lie algebra over $\mathbb{C}$.\\
For $i=1,2$, let $e_i \in \g$ be nilpotent elements and let $\chi_i$ be the associated linear forms constructed as in Subsection \ref{preliminW}. For $i=1,2$, we denote by  $(e_i,f_i,h_i)$ an $\mathfrak{sl}_2$-triple in which $e_i$ is embedded.\\
We adopt notations from Subsection \ref{preliminW} adding a subscript $\,_{(i)}$ for referring to the construction related to $e_i$.\\
For $i=1,2$, let $\g= \bigoplus_{j \in \Z} \g(j)_{(i)}$ be a good grading for $e_i$. By Remark \ref{semisimplegrading} the good grading $_{(i)}$ is induced by the adjoint action of a semisimple element $q_i$. Without loss of generality, we can assume that $q_1$ and $q_2$ belong to the same Cartan algebra $\mathfrak{h}$. \\
Throughout this subsection, we will make the following assumptions:
\begin{enumerate}
\label{assumption}
\item there is a direct sum decomposition $\mathfrak{m}_{\ell_{(2)}}=\mathfrak{m}_{\ell_{(1)}} \oplus \mathfrak{m}_0$, where $\mathfrak{m}_0$ is a $\mathfrak{h}$-stable Lie subalgebra of $\mathfrak{m}_{\ell_{(2)}}$ and $\mathfrak{m}_{\ell_{(1)}}$  is a Lie ideal of $\mathfrak{m}_{\ell_{(2)}}$,
\item the element $e_0:= e_2-e_1$ is nilpotent and  $e_0$ and the Lie algebra $\mathfrak{m}_0$ are contained in $\g(0)_{(1)}$ ,
\item the semisimple element $q_0:=q_2-q_1$ commutes with $e_1$.
\end{enumerate}
For simplicity, we denote the above assumptions with ($\diamondsuit$).\\
We stress that throughout this subsection the assumptions ($\diamondsuit$) are satisfied.\\\\
For introducing the functors $\text{Wh}_0$ and $Q_0 \otimes_{H_0}-$, we first recollect some results contained in \cite{GJ}. \\
By \cite[Lemma 3.2.1]{GJ}, there is an embedding of $\mathfrak{m}_0$ in $H_{\ell_{(1)}}$, which sends $y \in \mathfrak{m}_0$ to $y+I_{\ell_{(1)}}$. In consequence, we can define $I_0$ as the left $H_{\ell_{(1)}}$-ideal spanned by $(y-\chi_2(y))+I_{\ell_{(1)}}$, $y \in \mathfrak{m}_0$.\\ Let  $Q_0:=H_{\ell_{(1)}}/I_0$. The adjoint action of $\mathfrak{m}_0$ on $H_{\ell_{(1)}}$ descends to $Q_0$. Then, it makes sense to consider the subspace of invariants $H_0:=Q_0^{\text{ad}(\mathfrak{m}_0)}$.
This vector space turns out to be an algebra (see \cite[Lemma 3.2.2]{GJ}), where the multiplication is induced by the one in $H_{\ell_{(1)}}$.\\
Moreover, in virtue of \cite[Theorem 3.2.3]{GJ}, $H_0$ is isomorphic to $H_{\ell_{(2)}}$. \\ \\

For $i=1,2$, let $\C_i$ be the category of $U(\g)$-modules on which $\mathfrak{m}_{\ell_{(i)}}$ acts locally nilpotently. The category $\C_2$ is a subcategory of $\C_1$ because by assumptions $\mathfrak{m}_{\ell_{(1)}}\subseteq \mathfrak{m}_{\ell_{(2)}}$.\\
For $i=1,2$, we denote with $\m_i$ the category of Whittaker modules with respect to $e_i$; with $\Nn_i$ the category of $H_{\ell_{(i)}}$-modules; with $\trans_{(i)}$ the translation functor and with  $\text{Wh}_i$ the functor $\text{Wh}$ from Subsection \ref{preliminW} defined on $\m_i$.

Finally, by the embedding of $\mathfrak{m}_0$ in $H_{\ell_{(1)}}$, we can define the category $\m_0$, that is the category of finitely generated $H_{\ell_{(1)}}$-modules on which $y-\chi_2(y)$ acts locally nilpotently, for all $y \in \mathfrak{m}_0$. For $M \in \m_0$, we set
$$ \text{Wh}_0(M):=\{ m \in M \, |\, y.m=\chi_2(y)m \, \, \text{for all} \, \, y \in \mathfrak{m}_0\}.$$
By \cite[Lemma 5.2.1]{GJ}, for every $M \in \m_0$, the subspace $\text{Wh}_0(M)$ is a left $H_0$-module with action given by
$$(x+I_0).m=x.m, $$
for every $x+I_0 \in H_0$ and $m \in \text{Wh}_0(M)$.\\ As we recalled before, $H_0 \simeq H_{\ell_{(2)}}$ and hence $\text{Wh}_0(M)$ is a left $H_{\ell_{(2)}}$-module, for every $M \in \m_0$. This implies that, together with restriction on morphisms, $\text{Wh}_0$ gives a functor from the category $\m_0$ to the category $\Nn_2$.\\
In virtue of \cite[Lemma 5.2.2]{GJ}, we have that $Q_0$ is a $H_{\ell_{(1)}}$-$H_0$-bimodule, where $H_{\ell_{(1)}}$ acts on $Q_0$ via left multiplication, while $H_0$ acts on $Q_0$ via right multiplication. In consequence, for $N \in \Nn_2$ the $H_{\ell_{(1)}}$-module $Q_0 \otimes_{H_0} N$ is well defined. The action is given by:
$$(x).((y+I_0)\otimes_{H_0} n)=(xy+I_0) \otimes_{H_0}n, $$
for every $n \in N$ and $x,y \in H_{\ell_{(1)}}$. Furthermore, thanks to \cite[Lemma 5.2.2]{GJ},  the tensor product $Q_0 \otimes_{H_0} N$ lies in $\m_0$ for all $N \in \Nn_2$.

We are now in a position to state the following.
\begin{theorem}\cite[Theorem 5.2.3]{GJ}
\label{Skryabin_0}
Let $M \in \m_0$,  $N \in \Nn_2$. Then the functors $M \mapsto \text{Wh}_0(M)$ and $N \mapsto Q_0 \otimes_{H_0} N$ are quasi inverse equivalences.\\
Moreover, $\text{Wh}_1(-)$ and $Q_{\ell_{(1)}} \otimes_{H_{\ell_{(1)}}}(-)$ induce an equivalence of categories $\m_1 \simeq \m_0$ by restriction and we have $\text{Wh}_2= \text{Wh}_0 \circ \text{Wh}_1$. In particular, $\text{Wh}_0$  and $Q_0 \otimes_{H_0}$ are exact functors.
\end{theorem}
\begin{remark}
By \cite[Proposition 5.3.1]{GJ}, the natural isomorphism $$\text{Wh}_0 (Q_0 \otimes_{H_0}) \overset{\varepsilon_0}{\simeq} \text{id}$$ is given by
\begin{equation}
\label{epsilon0}
\begin{split}
(\varepsilon^{(0)}_M)^{-1} \colon M &\to \text{Wh}_0(Q_0 \otimes_{H_0} M)  \\
m &\mapsto 1 \otimes m,
\end{split}
\end{equation}
for any $M \in \m_0$ .
\end{remark}
\setcounter{equation}{0}
\section{Whittaker modules and $H_\ell$-modules as bimodule categories}
\label{chapter3}
We retain notation from Subsection \ref{preliminW} and we assume that $\ell$ is Lagrangian.\\
Our aim is to endow the category of $H_\ell$-modules with a bimodule structure over a category containing $U(\g)$-mod$_{\text{fin}}$. For this purpose, we firstly endow the category of Whittaker modules with a bimodule structure. Then, by means of Skryabin equivalence we will transport this structure to the category of $H_\ell$-modules.\\ \\
We denote by $\C_e$ the subcategory of $U(\g)$-mod on which $\mathfrak{m}_\ell$ acts locally nilpotently. 
\begin{remark}
\label{finitedimensionalmodulenil}
We stress that $\C_e$ contains the category $U(\g)$-$\text{mod}_{\text{fin}}$.
\end{remark}
We firstly show that $\C_e$ is a monoidal category. For this purpose, we need the following result.
\begin{lemma}
Let $\psi \colon \g \to \mathbb{C}$ be a linear form. Let $X,Z$ be in $U(\g)$-mod. Then,
\small
\begin{equation}
\label{formaz}
(y-\psi(y))^k.\bigl(\sum_{i=1}^{I}\sum_{j=1}^J \alpha_{i,j} x_i \otimes z_j \bigr)=\sum_{i=1}^I \sum_{j=1}^J \sum_{u=0}^{k} \alpha_{i,j} (y-\psi(y))^u.x_i \otimes y^{k-u}.z_j,
\end{equation}
for every $y \in \mathfrak{m}_\ell, x_i \in X, z_j \in Z$, $k \in \mathbb{N}$ and $\alpha_{i,j}\in \mathbb{C}$.
\end{lemma}
\begin{proof}
First of all, we can prove by induction that the following equation
\begin{equation}
\label{form1}
(y-\psi(y))^k.(x \otimes z)=\sum_{j=0}^k \binom{k}{j}(y-\psi(y))^j.x \otimes y^{k-j}.z,
\end{equation}
holds for every $y \in \mathfrak{m}_\ell, x \in X$ and $z \in Z$.
Finally, the linearity of the action gives Equation \eqref{formaz}.
\end{proof}
We are now in a position to prove the following.
\begin{lemma}
\label{strictC}
The category $\C_e$ endowed with the usual tensor product of modules is a monoidal category.
\end{lemma}
\begin{proof}
Since $U(\g)$-mod is monoidal, it is enough to verify that $\C_e$ is closed under the tensor product, i.e. for $X, Z \in \C_e$, we have to show that $\mathfrak{m}_\ell$ acts locally nilpotently on $X \otimes_\mathbb{C} Z.$ A generic element of $X \otimes Z$ is of the form
$$
\sum_{i=1}^{I}\sum_{j=1}^J \alpha_{i,j} x_i \otimes z_j ,
$$
for some $\alpha_{i,j} \in \mathbb{C}$, $ x_i \in X$ and $z_j \in Z$. Let $y \in \mathfrak{m}_\ell$. Substituting $\psi=0$ in Equation \eqref{formaz}, we have 
$$
y^k . \bigl(\sum_{i=1}^{I}\sum_{j=1}^J \alpha_{i,j} x_i \otimes z_j\bigr)= \sum_{i=1}^I \sum_{j=1}^J \sum_{u=0}^{k} \alpha_{i,j} y^u.x_i \otimes y^{k-u}.z_j.
$$
Since $\mathfrak{m}_\ell$ acts locally nilpotently on $X$ and $Z$, there exists a sufficient large $k \in \mathbb{N}$, such that one between $y^u.x_i$ and $y^{k-u}.z_j$ vanishes for any $u \in \{0, \dots, k\}$ and any $i,j$, concluding the proof.
\end{proof}
Recall that $\m$ stands for the category of Whittaker modules, that is the subcategory of $U(\g)$-mod on which $x-\chi(x)$ acts locally nilpotently for each $x \in \mathfrak{m}_\ell$. 
\begin{remark}
The category $\m$ does not contain finite-dimensional $U(\g)$-modules if $e \neq 0$. Indeed, let $V$ be a finite-dimensional $U(\g)$-module and suppose that $(x-\chi(x))$ acts locally nilpotently on $V$ for every $x \in \mathfrak{m}_\ell$. This assumption together with Remark \ref{finitedimensionalmodulenil} and with the fact that the elements $x$ and $x-\chi(x)$ commute for every $x \in \mathfrak{m}_\ell$ imply that $x-(x-\chi(x))=\chi(x)$ acts locally nilpotently on $V$. This condition holds if and only if $\chi(x)=0$ for every $x \in \mathfrak{m}_\ell$; this would contradict that  $(\cdot|\cdot)$ is non degenerate and $\g(-2) \subseteq \mathfrak{m}_\ell$. 
\end{remark}
We now define natural left and right actions of $\C_e$ on $U(\g)$-mod by means of the tensor product of modules $\otimes_\mathbb{C}$ in $U(\g)$-mod. We set: 
\begin{equation}
\begin{split}
\otimes_\m := \otimes_\mathbb{C}  \colon \C_e \times \m &\to U(\g)-\text{mod}\\
(-,\sim) &\mapsto - \otimes_\mathbb{C} \sim,
\end{split}
\end{equation}
and
\begin{equation}
\begin{split}
 \otimes^\m := \otimes_\mathbb{C}  \colon \m \times \C_e &\to U(\g)-\text{mod}\\
(-, \sim) &\mapsto - \otimes_\mathbb{C} \sim .
\end{split}
\end{equation}
From now on, we will write simply $\otimes$ to denote $\otimes_\mathbb{C}$.\\
We have the following result:
\begin{lemma}
The category $\m$ is a $(\C_e,\C_e)$-bimodule category, via the tensor product of $U(\g)$-modules, where $m_{X,Y,Z}, m_{X,Y,Z}^r$ and $b_{X,M,Y}$ are the shift of parentheses.
\end{lemma}
\begin{proof}
Firstly, we verify that that $M \otimes X$ is a Whittaker module for every Whittaker module $M$ and for every $U(\g)$-module $X$ on which $\mathfrak{m}_\ell$ acts locally nilpotently, that is we have to show that $y-\chi(y)$ acts locally nilpotently on $M \otimes X$, for every $M \in \m,  X \in \C_e$ and $y \in \mathfrak{m}_\ell$. This is equivalent to show that for every $z \in M \otimes X$, there exists $k_z \in \N$ such that $(y-\chi(y))^{k_z}.z=0$ for every $y \in \mathfrak{m}_\ell$. A generic element $z$ in $M \otimes X$ is of the form
$$ \sum_{i=1}^I \sum_{j=1}^J\alpha_{i,j} m_i \otimes x_j,$$
for some $\alpha_{i,j} \in \mathbb{C}$, $x_j \in X$ and $m_i \in M$. Then, our goal is to show that there exists $k_z \in \N$ such that
$$(y-\chi(y))^{k_z}.\bigl(\sum_{i=1}^I \sum_{j=1}^J\alpha_{i,j} m_i \otimes x_j\bigr)=0,$$
for every $y \in \mathfrak{m}_\ell$ and for $\alpha_{i,j}, m_i$ and $x_j$ as above.\\
Applying Equation \eqref{formaz}, we get:
$$
(y-\chi(y))^k.z=\sum_{i=1}^I \sum_{j=1}^J \sum_{u=0}^{k} \alpha_{i,j} (y-\chi(y))^u.m_i \otimes y^{k-u}.x_j
$$
for  $m_i \in M, x_j \in X$ and $\alpha_{i,j } \in \mathbb{C}$ as above and for every $y \in \mathfrak{m}_\ell$. Since $M$ is a Whittaker module and $X$ is a $U(\g)$-module on which $\mathfrak{m}_\ell$ acts locally nilpotently, there exists a sufficient large $K_z$, such that one between $(y-\chi(y))^u.m_i$ and $y^{K_z-u}.x_j$ vanishes for any $u \in \{0, \dots, K_z \}$ and any $i,j$. \\
The proof for the right action is analogous.\\ \\
Secondly, we have to verify that $\m$ is both a left and a right $\C_e$-module category. Since $m=m^r$ and the associativity constraint of the monoidal category $\C_e$ is identified with the identity as it happens for $U(\g)$-$\text{mod}$ (Lemma \ref{strictC}), the pentagon diagrams reduce to the trivial ones.\\
For similar reasons, diagrams \eqref{bim1} and \eqref{bim2}, become the trivial ones. Hence, we conclude that $\m$ is a $(\C_e,\C_e)$-bimodule category.
\end{proof}
Now, our aim is to transpose the $(\C_e,\C_e)$-bimodule structure of $\m$ to $\Nn$ by means of Skryabin equivalence, following the construction of Subsection \ref{transportofstruct}. \\
We define the right translation functor and the left translation functor as those given in Definition \ref{leftact} and Definition \ref{rightact}, specializing $F$ to $\text{Wh}$ and $G$ to $Q_\ell \otimes_{H_\ell}$.\\
More precisely, the definition of the right translation functor 
becomes:
\begin{equation}
\label{rightrans}
\begin{split}
\otimes^\Nn := \trans^r \colon \Nn \times \C_e &\to \Nn \\
(-,\sim) &\mapsto \text{Wh}((Q_\ell \otimes_{H_\ell} -) \otimes \sim),
\end{split}
\end{equation}
while  the left translation functor reads as
\begin{equation}
\label{lefttrans}
\begin{split}
\otimes_\Nn := \trans \colon \C_e \times \Nn &\to \Nn \\
(-,\sim) &\mapsto \text{Wh}(- \otimes (Q_\ell \otimes_{H_\ell} \sim)).
\end{split}
\end{equation}
Moreover, substituting $m=m^r= \text{id}$ and $F=\text{Wh}(-), G= Q_\ell \otimes_{H_\ell}-$ in equations \eqref{assconstraintleft} and \eqref{rightass}, we get, respectively, the left and the right module associativity constraints $n_{X,Y,N}$ and $n^r_{X,Y,N}$, for all $X, Y \in \C_e$ and $N \in \Nn$.\\ In particular, for $X,Y \in \C_e$ and $M \in \m$,  the module associativity constraint $$n_{X,Y,N}^r  \colon \text{Wh}((Q_\ell \otimes_{H_\ell}(\text{Wh}((Q_\ell \otimes_{H_\ell}M) \otimes X))\otimes Y) \to \text{Wh}((Q_\ell \otimes_{H_\ell} M)\otimes (X \otimes Y))$$ is the restriction of the following map 
\begin{equation}
\begin{split}
\label{leftassN}
\varphi \colon (Q_\ell \otimes_{H_\ell}(\text{Wh}((Q_\ell \otimes_{H_\ell}M) \otimes X))\otimes Y &\to (Q_\ell \otimes_{H_\ell} M)\otimes (X \otimes Y)\\
((u+I_\ell)(((u'+I_\ell)\otimes m)\otimes x))\otimes y &\mapsto u(((u'+I_\ell)\otimes m) \otimes x) \otimes y.
\end{split}
\end{equation}
\begin{remark}
Let  $U(\g)$-$\text{mod}_{\text{fin}}$ be the category of finite-dimensional $U(\g)$-modules. Then, $U(\g)$-$\text{mod}_{\text{fin}}$ is a subcategory of $\C_e$, which is closed under the tensor product. Hence, $U(\g)$-$\text{mod}_{\text{fin}}$ is a monoidal category.\\
If we restrict the definitions of the right translations functor \eqref{rightrans} and of the module associativity constraint $n^r$ to $U(\g)$-$\text{mod}_{\text{fin}}$, we get the definition for the translation functors given in \cite{G} and \cite{BK}.
\end{remark}
Furthermore, let $p_{X,N,Z}$ be  the compatibility isomorphisms defined as in \eqref{compatible}. \\
By Lemma \ref{lemmabimod}, we deduce the following:
\begin{theorem}
The category $\Nn$ endowed with $\trans$, $n_{X,Y,N}$, $\trans^r$, $n_{X,Y,N}^r$ and $p_{X,N,Z}$ defined above, is a $(\C_e,\C_e)$-bimodule category.
\end{theorem}
Moreover, by Lemma \ref{C-invariants} we obtain the following.
\begin{lemma}
\label{Whittakerinvariants}
Let $\varepsilon$ be the natural isomorphism \eqref{epsilon}, let  $s_{X,M}:=\text{Wh}(\text{id}_X \otimes \varepsilon ^{-1}_{M})$ and $t_{X,N}:=\varepsilon_{X \otimes (Q_\ell \otimes_{H_\ell} N)}$, for all $X \in \C_e, M \in \m$ and $N \in \Nn$. Then, the pairs $(\text{Wh}, s)$ and $(Q_\ell \otimes_{H_\ell},t)$ are $\C_e$-module functors.
\end{lemma}
\setcounter{equation}{0}
\section{$\C$-equivariant functors for reduction in stages}
\label{chapter4}
We retain notations from Subsection \ref{preliminariestage}. In particular, we stress that $\g$ is a finite-dimensional simple Lie algebra over $\mathbb{C}$ and that assumptions $(\diamondsuit)$ are satisfied.\\
Our final goal is to show that the equivalence introduced in Subsection \ref{preliminariestage} is invariant under the action of the category $\C_2$ and thus also of $U(\g)$-$\text{mod}_{\text{fin}}$.  

\begin{remark}
\label{C2modcat}
As we have shown in Section \ref{chapter3}, the category $\m_i$ is a $(\C_i,\C_i)$-bimodule category. Since $\C_2 \subseteq \C_1$, the category $\m_1$ is also a $(\C_2,\C_2)$-bimodule category. In addition, $\m_2$ is a $(\C_2,\C_2)$-bimodule subcategory of $\m_1$.
\end{remark}
We firstly need to understand the $\C_2$-module structure of $\m_0$.
\begin{lemma}
The category $\m_0$ is a $(\C_2, \C_2)$-bimodule category, where the left action bifunctor $\otimes_{\m_0}$ is $\trans_{(1)}\,_{|_{\m_0}}$, while the right one is  $\trans_{(1)}\,_{|_{\m_0}}^r$ .
\end{lemma}
\begin{proof}
Firstly, recall that $\m_2$ is a $(\C_2,\C_2)$-bimodule category as we observed in Remark \ref{C2modcat}.\\By Theorem \ref{Skryabin_0}, the functor $\text{Wh}_1$ establishes an equivalence between the categories $\m_2$ and $\m_0$. Hence, by transport of structure, we can endow $\m_0$ with a $\C_2$-module structure following the construction of Subsection \ref{chapter3}. In particular, this implies that the left action bifunctor is the restriction to $\m_0$ of $\trans_{(1)}$, while the right one is the restriction to $\m_0$ of  $\trans_{(1)}^r$. 
\end{proof}
The previous lemma together with Lemma \ref{C-invariants} gives the following corollary.
\begin{corollary}
The functor $\text{Wh}_1|_{\m_2}$ paired with  the natural isomorphism\\ $\text{Wh}_1|_{\m_2}(\id \otimes \varepsilon^{{(1)}^{-1}})$ is a $\C_2$-module functor, where $\varepsilon^{(1)}$ is the natural isomorphism defined in Equation \eqref{epsilon}. 
\end{corollary}
Now, we define the following natural isomorphisms  
\begin{equation*}
u_{X,N_2} \colon \text{Wh}_1 \circ (Q_{\ell_{(2)}}\otimes_{H_{\ell_{(2)}}} (X \trans_{(2)} N_2)) \to X \trans_{(1)}  \text{Wh}_1 \circ (Q_{\ell_{(2)}}\otimes_{H_{\ell_{(2)}}}  N_2)
\end{equation*}
as 
\begin{equation}
\label{u,Q0}
u_{X,N_2}:=\text{Wh}_1 \bigl( (\text{id}_X \otimes {\varepsilon^{(1)^{-1}}_{Q_{\ell_{(2)}} \otimes_{H_{\ell_{(2)}}}N_2}}) \circ \varepsilon^{(2)}_{X \otimes (Q_{\ell_{(2)}}\otimes_{H_{\ell_{(2)}} }N_2)}\bigr)
\end{equation}
and
\begin{equation*}
v_{X,M_0} \colon \text{Wh}_2 \circ (Q_{\ell_{(1)}}\otimes_{H_{\ell_{(1)}}} (X \trans_{(1)} M_0)) \to X \trans_{(2)}  \text{Wh}_2 \circ (Q_{\ell_{(1)}}\otimes_{H_{\ell_{(1)}}}  M_0)
\end{equation*}
as 
\begin{equation}
\label{v,Wh0}
v_{X,M_0}:=\text{Wh}_2 \bigl( (\text{id}_X \otimes {\varepsilon^{(2)^{-1}}_{Q_{\ell_{(1)}} \otimes_{H_{\ell_{(1)}}}M_0}}) \circ \varepsilon^{(1)}_{X \otimes (Q_{\ell_{(1)}}\otimes_{H_{\ell_{(1)}} }M_0)}\bigr).
\end{equation}
for every $X \in \C_2, N_2 \in \Nn_2$ and $M_0 \in \m_0$ and for $\varepsilon$ the natural isomorphism defined in Remark \ref{epsilon}.\\
We are now in a position to prove the main result. 
\begin{theorem}
Let $u$ and $v$ be the natural isomorphisms defined respectively in \eqref{u,Q0} and \eqref{v,Wh0}. Then, the pairs $(Q_0 \otimes_{H_0}, u)$ and $(\text{Wh}_0, v)$ are mutually inverse $\C_2$-module equivalences. 

\end{theorem}
\begin{proof}
We prove the Theorem for the functor $Q_0 \otimes_{H_0}$.\\
Theorem \ref{Skryabin_0} implies that the following diagram is commutative
\begin{center}
\small
\begin{equation*}
\begin{tikzpicture}
  \matrix (m) [matrix of math nodes,rotate=90,transform shape,row sep=6em,column sep=12em,minimum width=2em]
  {
\m_2 & \Nn_2 \\
\,  & \mathcal{M}_0.\\
 };

  \path[-stealth]

(m-1-2) edge node [above] {$\scriptstyle{Q_{\ell_{(2)}}\otimes_{H_{\ell_{(2)}}}}$} (m-1-1)
(m-1-1) edge node [below] {$\scriptstyle{ \text{Wh}_1|_{\m_2} \, \qquad}$}  (m-2-2)
(m-1-2) edge node [right] {$\scriptstyle{Q_0 \otimes_{H_0}}$}  (m-2-2)  
;
\end{tikzpicture}
\end{equation*}
\end{center}
Moreover, by Lemma \ref{Whittakerinvariants} the functors $\text{Wh}_1$ and $Q_{\ell_{(2)}}\otimes_{H_{\ell_{(2)}}}$ paired respectively with the natural isomorphisms $\text{Wh}_1(\text{id}_- \otimes \varepsilon ^{{(1)}^{-1}}_-)$ and $\varepsilon^{(2)}_{- \otimes (Q_{\ell_{(2)}} \otimes_{H_{\ell_{(2)}}} -)}$, are $\C_2$-module functors. \\Finally, since $Q_0 \otimes_{H_0}=Q_{\ell_{(2)}}\otimes_{H_{\ell_{(2)}}}\circ \text{Wh}_1$, we can apply Lemma \ref{compositionfunctor} obtaining that $(Q_0 \otimes_{H_0},u)$ is a $\C_2$-module functor. \\
The proof for the functor $\text{Wh}_0$ is analogous.
\end{proof}

\setcounter{section}{0}
\renewcommand{\thesection}{\Alph{section}}
\renewcommand{\theHsection}{\Alph{section}}
\setcounter{equation}{0}
\section{Appendix}
\label{appendix}
In this section we exibhit in full details the proof of Lemma \ref{left} and of Lemma \ref{lemmabimod}.\\ \\
\emph{Proof of Lemma \ref{left}} We were left to show that the following equation 
\small
\begin{equation}
\label{pentagA}
\begin{split}
&F\bigl((\text{id}_X \otimes_\m GF((\text{id}_Y \otimes_\m \varepsilon^{-1}_{Z \otimes_\m G(N)})\circ m_{Y,Z,G(N)})) \circ (\text{id}_X \otimes_\m \varepsilon^{-1}_{(Y \otimes_\C Z) \otimes_\m G(N)}) \bigr)\circ
\\
&\circ F \bigl( m_{X,Y \otimes_\C Z, G(N)} \circ (a_{X,Y,Z} \otimes_\m \text{id}_{G(N)})\bigr)=\\
&F \hspace{-2pt}\bigl(\hspace{-2pt}(\text{id}_X \otimes_\m \varepsilon^{-1}_{Y \otimes_\m (GF(Z \otimes_\m G(N)))})\hspace{-2pt} \circ\hspace{-2pt}  m_{X,Y,GF(Z \otimes_\m G(N))} \hspace{-2pt}\circ \hspace{-2pt}  (\text{id}_X \otimes_\C \text{id}_Y \otimes_\m \varepsilon^{-1}_{Z \otimes_\m G(N)}) \bigr)\circ \\
&\circ  F( m_{X \otimes_\C Y,Z,G(N)})
\end{split}
\end{equation}
holds for any $X,Y,Z \in \C$ and $N \in \Nn$.
We recall that we needed the following commutative diagrams:
\begin{center}
\begin{equation}
\label{AA}
\begin{tikzpicture}
  \matrix (m) [matrix of math nodes,row sep=3em,column sep=6em,minimum width=2em]
  {
    GF((Y \otimes_\C Z)\otimes_\m G(N)) & GF(Y \otimes_\m (Z \otimes_\m G(N)))\\
(Y \otimes_\C Z)\otimes_\m G(N)&Y \otimes_\m (Z \otimes_\m G(N)) , \\ };
  \path[-stealth]
    (m-1-1) edge node [above] {$\scriptstyle{GF(m_{Y,Z,G(N)})}$} (m-1-2)
    (m-1-1) edge node [left] {$\scriptstyle{\varepsilon_{(Y \otimes_\C Z)\otimes_\m G(N)}}$} (m-2-1)
    (m-2-1) edge node [above] {$\scriptstyle{m_{Y,Z,G(N)}}$} (m-2-2)
(m-1-2) edge node [right] {$\scriptstyle{\varepsilon_{Y \otimes_\m (Z \otimes_\m G(N))}}$} (m-2-2);
\end{tikzpicture}
\end{equation}
\end{center}

\begin{center}
\begin{equation}
\label{CA}
\begin{tikzpicture}[thick,scale=0.6, every node/.style={scale=0.95}]
  \matrix (m) [matrix of math nodes,row sep=3em,column sep=6em,minimum width=2em]
  {
    GF(Y \otimes_\m ( Z\otimes_\m G(N)))& Y\otimes_\m (Z \otimes_\m G(N))\\
GF(Y\otimes_\m GF (Z \otimes_\m G(N)))&Y\otimes_\m GF (Z \otimes_\m G(N)), \\ };
  \path[-stealth]
    (m-1-1) edge node [above] {$\scriptstyle{\varepsilon_{Y \otimes_\m (Z \otimes_\m G(N))}}$} (m-1-2)
    (m-1-1) edge node [left] {$\scriptstyle{GF(\text{id}_Y \otimes_\m \varepsilon^{-1}_{Z \otimes_\m G(N)})}$} (m-2-1)
    (m-2-1) edge node [above] {$\scriptstyle{\varepsilon_{Y \otimes_\m GF(Z \otimes_\m G(N))}}$} (m-2-2)
(m-1-2) edge node [right] {$\scriptstyle{\text{id}_Y \otimes_\m \varepsilon^{-1}_{Z \otimes_\m G(N)}}$} (m-2-2);
\end{tikzpicture}
\end{equation}
\end{center}
and
 \begin{center}
\small
\begin{equation}
\label{BA}
\begin{tikzpicture}
  \matrix (m) [matrix of math nodes,row sep=3em,column sep=8em,minimum width=2em]
  {
    (X\otimes_\C Y) \otimes_\m GF(Z \otimes_\m G(N))& (X\otimes_\C Y) \otimes_\m (Z \otimes_\m G(N))\\
X\otimes_\m( Y \otimes_\m GF(Z \otimes_\m G(N)))&X\otimes_\m( Y \otimes_\m (Z \otimes_\m G(N))), \\ };
  \path[-stealth]
    (m-1-1) edge node [above] {$\scriptstyle{\text{id}_{X \otimes_\C Y} \otimes_\m \varepsilon_{Z \otimes_\m G(N)}}$} (m-1-2)
    (m-1-1) edge node [left] {$\scriptstyle{m_{X,Y,GF(Z \otimes_\m G(N))}}$} (m-2-1)
    (m-2-1) edge node [above] {$\scriptstyle{\text{id}_X \otimes_\m(\text{id}_Y \otimes_\m \varepsilon_{Z \otimes_\m G(N)})}$} (m-2-2)
(m-1-2) edge node [right] {$\scriptstyle{m_{X,Y,Z \otimes_\m G(N)}}$} (m-2-2);
\end{tikzpicture}
\end{equation}
\end{center}
for $X, Y , Z \in \C$ and $N \in \Nn$.\\
Let us consider the left hand-side of equation \eqref{pentagA}.
The commutativity of diagram \eqref{AA} gives:
\small
\begin{equation*}
\begin{split}
&F\bigl((\text{id}_X \otimes_\m GF((\text{id}_Y \otimes_\m \varepsilon^{-1}_{Z \otimes_\m G(N)})\circ m_{Y,Z,G(N)})) \circ (\text{id}_X \otimes_\m \varepsilon^{-1}_{(Y \otimes Z) \otimes_\m G(N)}) \bigr)\circ
\\
&\circ F \bigl(m_{X,Y \otimes_\C Z, G(N)} \circ (a_{X,Y,Z} \otimes_\m \text{id}_{G(N)})\bigr)\\
&=F\bigl((\text{id}_X \otimes_\m GF(\text{id}_Y \otimes_\m \varepsilon^{-1}_{Z \otimes_\m G(N)})\circ \varepsilon_{Y\otimes_\m (Z \otimes_\m G(N)))}^{-1}) \circ (\text{id}_X \otimes_\m   m_{Y,Z,G(N)} ) \bigr) \circ \\
& \circ F \bigl(m_{X,Y \otimes_\C Z, G(N)} \circ (a_{X,Y,Z} \otimes_\m \text{id}_{G(N)})\bigr).
\end{split}
\end{equation*}
Finally, by the commutativity of diagram \eqref{CA}, the above term reads as:
\small
\begin{equation}
\label{lsA}
\begin{split}
&F\bigl((\text{id}_X \otimes_\m (\varepsilon^{-1}_{Y \otimes_\m (GF(Z \otimes_\m G(N)))} \circ (\text{id}_Y \otimes_\m \varepsilon^{-1}_{Z \otimes_\m G(N)})))\circ (\text{id}_X \otimes_\m  m_{Y,Z,G(N)} )\bigr ) \circ \\
&\circ F\bigl(m_{X,Y \otimes_\C Z, G(N)} \circ (a_{X,Y,Z} \otimes_\m \text{id}_{G(N)})\bigr).
\end{split}
\end{equation}
Consider now the right hand-side of equation \eqref{pentagA}.
By the commutativity of diagram \eqref{BA}, the right hand-side becomes:
\small
\begin{equation}
\label{rsA}
\begin{split}
&F\bigl((\text{id}_X \otimes_\m (\varepsilon^{-1}_{Y \otimes_\m GF((Z \otimes_\m G(N)))} \circ (\text{id}_Y \otimes_\m \varepsilon^{-1}_{Z \otimes_\m G(N)})))\circ m_{X,Y,Z \otimes_\m G(N)}) \bigr) \circ \\
&\circ F( m_{X \otimes_\C Y,Z,G(N)}).
\end{split}
\end{equation}
Since $F\bigl((\text{id}_X \otimes_\m (\varepsilon^{-1}_{Y \otimes_\m (GF(Z \otimes_\m G(N)))} \circ (\text{id}_Y \otimes_\m \varepsilon^{-1}_{Z \otimes_\m G(N)})))$ is an isomorphism, expression \eqref{lsA} and expression \eqref{rsA} are equal if and only if
\begin{equation*}
\small
\begin{split}
&F((\text{id}_X \otimes_\m m_{Y,Z,G(N)}) \circ m_{X,Y\otimes_\C Z, G(N)} \circ (a_{X,Y,Z} \otimes_\m \text{id}_{G(N)}))=\\&F(m_{X,Y,Z \otimes_\m G(N)} \circ m_{X \otimes_\C Y, Z, G(N)}).
\end{split}
\end{equation*}
This equation holds, since $\m$ is a left $\C$-module category. \\ \qed \\
\emph{Proof of Lemma \ref{lemmabimod}.} We had to verify that the following equation
\small
\begin{equation}
\label{bimodA}
\begin{split}
&F(\text{id}_X \otimes_\m GF((\text{id}_Y \otimes_\m \varepsilon^{-1}_{G(N) \otimes^\m Z}) \circ b_{Y,G(N),Z} \circ(\varepsilon_{Y \otimes_\m G(N)} \otimes^\m \text{id}_Z))) \circ \\
&F((\text{id}_X \otimes_\m \varepsilon^{-1}_{GF(Y \otimes_\m G(N))\otimes^\m \text{id}_Z}) \hspace{-2pt}\circ \hspace{-2pt}b_{X,GF(Y \otimes_\m G(N)),Z}\hspace{-2pt} \circ \hspace{-2pt} (\varepsilon_{X \otimes_\m GF(Y \otimes_\m G(N))} \otimes^\m \text{id}_Z)) \circ \\
&F(GF((\text{id}_X \otimes_\m \varepsilon^{-1}_{Y \otimes_\m G(N)}) \circ m_{X,Y,G(N)}) \otimes^\m \text{id}_Z)=\\
&F((\text{id}_X \otimes_\m \varepsilon^{-1}_{Y \otimes_\m GF(G(N) \otimes^\m Z)}) \circ m_{X,Y,GF(G(N) \otimes^\m Z)} \circ (id_{X \otimes_\C Y} \otimes_\m \varepsilon^{-1}_{G(N) \otimes^\m Z})) \circ \\
&F(b_{X \otimes_\C Y,G(N),Z} \circ (\varepsilon_{(X \otimes_\C Y) \otimes_\m G(N)} \otimes^\m \text{id}_Z)).
\end{split}
\end{equation}
holds for any $X, Y \in \C$, $Z \in \mathcal{D}$ and $N \in \Nn$.\\
We recall the commutative diagrams we needed for the proof. Specifically, for $X, Y, Z \in \C$ and $N \in \Nn$ we had
\begin{center}
\begin{equation}
\label{1A}
\begin{tikzpicture}
  \matrix (m) [matrix of math nodes,row sep=3em,column sep=6em,minimum width=2em]
  {
    GF(X \otimes_\m (Y \otimes_\m G(N))) & X \otimes_\m (Y \otimes_\m G(N))\\
GF((X \otimes_\C Y) \otimes_\m G(N)))&(X \otimes_\C Y) \otimes_\m G(N)), \\ };
  \path[-stealth]
    (m-1-1) edge node [above] {$\scriptstyle{\varepsilon_{X \otimes_\m (Y \otimes_\m G(N))}}$} (m-1-2)
    (m-2-1) edge node [left] {$\scriptstyle{GF(m_{X,Y,G(N)})}$} (m-1-1)
    (m-2-1) edge node [above] {$\scriptstyle{\varepsilon_{(X \otimes_\C Y) \otimes_\m G(N)}}$} (m-2-2)
(m-2-2) edge node [right] {$\scriptstyle{m_{X,Y,G(N)}}$} (m-1-2);
\end{tikzpicture}
\end{equation}
\end{center}

\begin{center}
\begin{equation}
\label{2A}
\begin{tikzpicture}
  \matrix (m) [matrix of math nodes,row sep=3em,column sep=6em,minimum width=2em]
  {
    GF(X \otimes_\m GF(Y \otimes_\m G(N))) & X \otimes_\m GF(Y \otimes_\m G(N))\\
  GF(X \otimes_\m (Y \otimes_\m G(N)))&X \otimes_\m (Y \otimes_\m G(N)), \\ };
  \path[-stealth]
    (m-1-1) edge node [above] {$\scriptstyle{\varepsilon_{X \otimes_\m GF(Y \otimes_\m G(N))}}$} (m-1-2)
    (m-1-1) edge node [left] {$\scriptstyle{GF(\text{id}_X \otimes_\m \varepsilon_{Y \otimes_\m G(N)})}$} (m-2-1)
    (m-2-1) edge node [above] {$\scriptstyle{\varepsilon_{X \otimes_\m (Y \otimes_\m G(N))}}$} (m-2-2)
(m-1-2) edge node [right] {$\scriptstyle{\text{id}_X \otimes_\m \varepsilon_{Y \otimes_\m G(N)}}$} (m-2-2);
\end{tikzpicture}
\end{equation}
\end{center}

\begin{center}
\begin{equation}
\label{3A}
\begin{tikzpicture}
  \matrix (m) [matrix of math nodes,row sep=3em,column sep=6em,minimum width=2em]
  {
    GF((Y \otimes_\m G(N)) \otimes^\m Z) & (Y \otimes_\m G(N)) \otimes^\m Z\\
  GF(Y \otimes_\m (G(N) \otimes^\m Z))&Y \otimes_\m (G(N) \otimes^\m Z), \\ };
  \path[-stealth]
    (m-1-1) edge node [above] {$\scriptstyle{\varepsilon_{(Y \otimes_\m G(N)) \otimes^\m Z}}$} (m-1-2)
    (m-1-1) edge node [left] {$\scriptstyle{GF(b_{Y,G(N),Z})}$} (m-2-1)
    (m-2-1) edge node [above] {$\scriptstyle{\varepsilon_{Y \otimes_\m (G(N) \otimes^\m Z)}}$} (m-2-2)
(m-1-2) edge node [right] {$\scriptstyle{b_{Y,G(N),Z}}$} (m-2-2);
\end{tikzpicture}
\end{equation}
\end{center}
\begin{center}
\begin{equation}
\label{4A}
\begin{tikzpicture}[thick,scale=0.6, every node/.style={scale=0.9}]
  \matrix (m) [matrix of math nodes,row sep=3em,column sep=6em,minimum width=2em]
  {
    GF(GF(Y \otimes_\m G(N)) \otimes^\m Z) &GF(Y \otimes_\m G(N)) \otimes^\m Z\\
  GF((Y \otimes_\m G(N)) \otimes^\m Z)&(Y \otimes_\m G(N)) \otimes^\m Z,\\ };
  \path[-stealth]
    (m-1-1) edge node [above] {$\scriptstyle{\varepsilon_{GF(Y \otimes_\m G(N)) \otimes^\m Z}}$} (m-1-2)
    (m-1-1) edge node [left] {$\scriptstyle{GF(\varepsilon_{Y \otimes_\m G(N)} \otimes^\m \text{id}_Z)}$} (m-2-1)
    (m-2-1) edge node [above] {$\scriptstyle{\varepsilon_{(Y \otimes_\m G(N)) \otimes^\m Z}}$} (m-2-2)
(m-1-2) edge node [right] {$\scriptstyle{\varepsilon_{Y \otimes_\m G(N)} \otimes^\m \text{id}_Z}$} (m-2-2);
\end{tikzpicture}
\end{equation}
\end{center}

\begin{center}
\begin{equation}
\label{5A}
\begin{tikzpicture}[thick,scale=0.6, every node/.style={scale=0.9}]
  \matrix (m) [matrix of math nodes,row sep=3em,column sep=6em,minimum width=2em]
  {
    GF(Y \otimes_\m (G(N) \otimes^\m Z)) &Y \otimes_\m (G(N) \otimes^\m Z)\\
  GF(Y \otimes_\m GF(G(N) \otimes^\m Z))&Y \otimes_\m GF(G(N) \otimes^\m Z),\\ };
  \path[-stealth]
    (m-1-1) edge node [above] {$\scriptstyle{\varepsilon_{Y \otimes_\m (G(N) \otimes^\m Z)}}$} (m-1-2)
    (m-1-1) edge node [left] {$\scriptstyle{GF(\text{id}_Y \otimes_\m \varepsilon^{-1}_{G(N) \otimes^\m Z})}$} (m-2-1)
    (m-2-1) edge node [above] {$\scriptstyle{\varepsilon_{Y \otimes_\m GF(G(N) \otimes^\m Z)}}$} (m-2-2)
(m-1-2) edge node [right] {$\scriptstyle{\text{id}_Y \otimes_\m \varepsilon^{-1}_{G(N) \otimes^\m Z}}$} (m-2-2);
\end{tikzpicture}
\end{equation}
\end{center}

\begin{center}
\begin{equation}
\label{6A}
\begin{tikzpicture}[thick,scale=0.6, every node/.style={scale=0.9}]
  \matrix (m) [matrix of math nodes,row sep=3em,column sep=6em,minimum width=2em]
  {
    (X \otimes_\C Y) \otimes_\m GF(G(N) \otimes^\m Z) & X \otimes_\m (Y \otimes_\m GF(G(N) \otimes^\m Z)) \\
  (X \otimes_\C Y) \otimes_\m (G(N) \otimes^\m Z)& X \otimes_\m (Y \otimes_\m (G(N) \otimes^\m Z))\\ };
  \path[-stealth]
    (m-1-1) edge node [above] {$\scriptstyle{m_{X,Y, GF(G(N) \otimes^\m Z)}}$} (m-1-2)
    (m-1-1) edge node [left] {$\scriptstyle{\text{id}_{X \otimes_\C Y} \otimes_\m \varepsilon_{G(N) \otimes^\m Z}}$} (m-2-1)
    (m-2-1) edge node [above] {$\scriptstyle{m_{X,Y,G(N) \otimes^\m Z}}$} (m-2-2)
(m-1-2) edge node [right] {$\scriptstyle{\text{id}_ X \otimes_\m (\text{id}_Y \otimes_\m \varepsilon_{G(N) \otimes^\m Z})}$} (m-2-2);
\end{tikzpicture}
\end{equation}
\end{center}
and 
\begin{center}
\begin{equation}
\label{7A}
\begin{tikzpicture}[thick,scale=0.6, every node/.style={scale=0.85}]
  \matrix (m) [matrix of math nodes,row sep=3em,column sep=6em,minimum width=2em]
  {
    (X \otimes_\m GF(Y \otimes_\m G(N))) \otimes^\m Z & X \otimes_\m (GF(Y \otimes_\m G(N)) \otimes^\m Z) \\
  (X \otimes_\m (Y \otimes_\m G(N))) \otimes^\m Z& X \otimes_\m ((Y \otimes_\m G(N)) \otimes^\m Z).\\ };
  \path[-stealth]
    (m-1-1) edge node [above] {$\scriptstyle{b_{X, GF(Y \otimes_\m G(N)),Z}}$} (m-1-2)
    (m-1-1) edge node [left] {$\scriptstyle{(\text{id}_{X} \otimes_\m \varepsilon_{Y \otimes_\m G(N)}) \otimes^\m \text{id}_Z}$} (m-2-1)
    (m-2-1) edge node [above] {$\scriptstyle{b_{X, Y \otimes_\m G(N),Z}}$} (m-2-2)
(m-1-2) edge node [right] {$\scriptstyle{\text{id}_{X} \otimes_\m (\varepsilon_{Y \otimes_\m G(N)} \otimes^\m \text{id}_Z)}$} (m-2-2);
\end{tikzpicture}
\end{equation}
\end{center}

Consider the left-hand side of equation \eqref{bimodA}. By commutativity of diagram \eqref{1A}, we get:
\begin{equation*}
\begin{split}
&F(\text{id}_X \otimes_\m GF((\text{id}_Y \otimes_\m \varepsilon^{-1}_{G(N) \otimes^\m Z}) \circ b_{Y,G(N),Z} \circ(\varepsilon_{Y \otimes_\m G(N)} \otimes^\m \text{id}_Z))) \circ \\
&F((\text{id}_X \otimes_\m \varepsilon^{-1}_{GF(Y \otimes_\m G(N))\otimes^\m \text{id}_Z}) \circ b_{X,GF(Y \otimes_\m G(N)),Z} \circ (\varepsilon_{X \otimes_\m GF(Y \otimes_\m G(N))} \otimes^\m \text{id}_Z)) \circ \\
&F(GF((\text{id}_X \otimes_\m \varepsilon^{-1}_{Y \otimes_\m G(N)}) \circ m_{X,Y,G(N)}) \otimes^\m \text{id}_Z)=\\
&F(\text{id}_X \otimes_\m GF((\text{id}_Y \otimes_\m \varepsilon^{-1}_{G(N) \otimes^\m Z}) \circ b_{Y,G(N),Z} \circ(\varepsilon_{Y \otimes_\m G(N)} \otimes^\m \text{id}_Z))) \circ \\
&F((\text{id}_X \otimes_\m \varepsilon^{-1}_{GF(Y \otimes_\m G(N))\otimes^\m \text{id}_Z}) \circ b_{X,GF(Y \otimes_\m G(N)),Z} \circ (\varepsilon_{X \otimes_\m GF(Y \otimes_\m G(N))} \otimes^\m \text{id}_Z)) \circ \\
&F((GF(\text{id}_X \otimes_\m \varepsilon^{-1}_{Y \otimes_\m G(N)}) \otimes^\m \text{id}_Z) \circ ((\varepsilon^{-1}_{X \otimes_\m (Y \otimes_\m G(N))} \circ m_{X,Y,G(N)}) \otimes^\m \text{id}_Z) \circ\\
&F(\varepsilon_{(X \otimes_\C Y) \otimes_\m G(N)} \otimes^\m \text{id}_Z)).
\end{split}
\end{equation*}
The commutativity of diagram \eqref{2A} gives:
\begin{equation*}
\begin{split}
&F(\text{id}_X \otimes_\m GF((\text{id}_Y \otimes_\m \varepsilon^{-1}_{G(N) \otimes^\m Z}) \circ b_{Y,G(N),Z} \circ(\varepsilon_{Y \otimes_\m G(N)} \otimes^\m \text{id}_Z))) \circ \\
&F((\text{id}_X \otimes_\m \varepsilon^{-1}_{GF(Y \otimes_\m G(N))\otimes^\m \text{id}_Z}) \circ b_{X,GF(Y \otimes_\m G(N)),Z} \circ (\varepsilon_{X \otimes_\m GF(Y \otimes_\m G(N))} \otimes^\m \text{id}_Z)) \circ \\
&F((\varepsilon^{-1}_{X \otimes_\m GF(Y \otimes_\m G(N))} \circ(\text{id}_X \otimes_\m \varepsilon_{Y \otimes_\m G(N)}^{-1}) \circ m_{X,Y,G(N)} \circ \varepsilon_{(X \otimes_\C Y) \otimes_\m G(N)})\otimes^\m \text{id}_Z)=\\
&F(\text{id}_X \otimes_\m GF((\text{id}_Y \otimes_\m \varepsilon^{-1}_{G(N) \otimes^\m Z}) \circ b_{Y,G(N),Z} \circ(\varepsilon_{Y \otimes_\m G(N)} \otimes^\m \text{id}_Z))) \circ \\
&F((\text{id}_X \otimes_\m \varepsilon^{-1}_{GF(Y \otimes_\m G(N))\otimes^\m \text{id}_Z}) \circ b_{X,GF(Y \otimes_\m G(N)),Z}) \circ \\
&F(((\text{id}_X \otimes_\m \varepsilon_{Y \otimes_\m G(N)}^{-1}) \circ m_{X,Y,G(N)} \circ \varepsilon_{(X \otimes_\C Y) \otimes_\m G(N)})\otimes^\m \text{id}_Z).
\end{split}
\end{equation*}
By commutativity of \eqref{7A}, the previous term becomes:
\small
\begin{equation*}
\begin{split}
&F(\text{id}_X \otimes_\m GF((\text{id}_Y \otimes_\m \varepsilon^{-1}_{G(N) \otimes^\m Z}) \circ b_{Y,G(N),Z} \circ(\varepsilon_{Y \otimes_\m G(N)} \otimes^\m \text{id}_Z))) \circ \\
&F((\text{id}_X \otimes_\m \varepsilon^{-1}_{GF(Y \otimes_\m G(N))\otimes^\m \text{id}_Z}) \circ (\text{id}_X \otimes_\m (\varepsilon^{-1}_{Y \otimes_\m G(N)} \otimes^\m \text{id}_Z)) \circ  b_{X,Y \otimes_\m G(N),Z}) \circ \\
&F(( m_{X,Y,G(N)} \circ \varepsilon_{(X \otimes_\C Y) \otimes_\m G(N)})\otimes^\m \text{id}_Z).
\end{split}
\end{equation*}
By virtue of commutativity of diagram \eqref{3A},  the above reads as:
\small
\begin{equation*}
\begin{split}
&F(\text{id}_X \otimes_\m (GF(\text{id}_Y \otimes_\m \varepsilon^{-1}_{G(N) \otimes^\m Z}) \circ (\varepsilon^{-1}_{Y \otimes_\m (G(N) \otimes^\m Z)} \circ b_{Y,G(N),Z} \circ \varepsilon_{(Y \otimes_\m G(N))\otimes^\m Z}))) \circ\\
&F(GF(\varepsilon_{Y \otimes_\m G(N)} \otimes^\m \text{id}_Z)\circ (\text{id}_X \otimes_\m \varepsilon^{-1}_{GF(Y \otimes_\m G(N))\otimes^\m \text{id}_Z})) \circ\\
&F((\text{id}_X \otimes_\m (\varepsilon^{-1}_{Y \otimes_\m G(N)} \otimes^\m \text{id}_Z)) \circ  b_{X,Y \otimes_\m G(N),Z}) \circ \\
&F(( m_{X,Y,G(N)} \circ \varepsilon_{(X \otimes_\C Y) \otimes_\m G(N)})\otimes^\m \text{id}_Z).
\end{split}
\end{equation*}
By commutativity of diagram \eqref{4A} the previous term becomes:
\small
\begin{equation*}
\begin{split}
&F(\text{id}_X \otimes_\m (GF(\text{id}_Y \otimes_\m \varepsilon^{-1}_{G(N) \otimes^\m Z}) \circ (\varepsilon^{-1}_{Y \otimes_\m (G(N) \otimes^\m Z)} \circ b_{Y,G(N),Z} )))\circ \\
&F( b_{X,Y \otimes_\m G(N),Z}\circ (( m_{X,Y,G(N)} \circ \varepsilon_{(X \otimes_\C Y) \otimes_\m G(N)})\otimes^\m \text{id}_Z)) .
\end{split}
\end{equation*}
Finally, by commutativity of diagram \eqref{5A} the above equals
\small
\begin{equation}
\label{sxbim}
\begin{split}
&F(\text{id}_X \otimes_\m (\varepsilon^{-1}_{Y \otimes_\m GF(G(N) \otimes^\m Z)} \circ (\text{id}_Y \otimes_\m \varepsilon^{-1}_{G(N) \otimes^\m Z}) \circ b_{Y,G(N),Z} ))\circ \\
&F( b_{X,Y \otimes_\m G(N),Z}\circ (( m_{X,Y,G(N)} \circ \varepsilon_{(X \otimes_\C Y) \otimes_\m G(N)})\otimes^\m \text{id}_Z)) .
\end{split}
\end{equation}
Consider now the right hand-side of equation \eqref{bimodA}. By commutativity of diagram \eqref{6A}, it equals
\small
\begin{equation}
\label{dxbim}
\begin{split}
&F((\text{id}_X \otimes_\m \varepsilon^{-1}_{Y \otimes_\m GF(G(N) \otimes^\m Z)}) \circ (\text{id}_X \otimes_\m (\text{id}_Y \otimes_\m \varepsilon^{-1}_{G(N) \otimes^\m Z}))\circ m_{X,Y,G(N) \otimes^\m Z}) \circ \\
&F( b_{X \otimes_\C Y,G(N),Z} \circ (\varepsilon_{(X \otimes_\C Y) \otimes_\m G(N)} \otimes^\m \text{id}_Z)).
\end{split}
\end{equation}
Since $F((\text{id}_X \otimes_\m \varepsilon^{-1}_{Y \otimes_\m GF(G(N) \otimes^\m Z)}) \circ (\text{id}_X \otimes_\m (\text{id}_Y \otimes_\m \varepsilon^{-1}_{G(N) \otimes^\m Z})))$ and \vspace{0.2cm}\\$(\varepsilon_{(X \otimes_\C Y) \otimes_\m G(N)} \otimes^\m \text{id}_Z)$ are natural isomorphisms, then equation \eqref{bimodA} is satisfied if and only if :
\begin{equation*}
\begin{split}
&F((\text{id}_X \otimes_\m  b_{Y,G(N),Z} )\circ  b_{X,Y \otimes_\m G(N),Z} \circ ( m_{X,Y,G(N)} \otimes^\m \text{id}_Z))=\\
&F( m_{X,Y,G(N) \otimes^\m Z} \circ  b_{X \otimes_\C Y,G(N),Z}).
\end{split}
\end{equation*}
This equation holds, since $\m$ is a $(\C, \D)$-bimodule category, concluding the proof. \\ \qed

\end{document}